\documentclass[11pt,fleqn]{iopart}

\usepackage{iopams}
\usepackage{fouriernc}
\usepackage{bm}

\expandafter\let\csname equation*\endcsname\relax

\expandafter\let\csname endequation*\endcsname\relax

\usepackage{amsmath}
\usepackage{amsthm}
\newtheorem{Proposition}{Proposition}
\newtheorem{Corollary}[Proposition]{Corollary}
\theoremstyle{remark}
\newtheorem{Remark}{Remark}

\usepackage[breaklinks=true,colorlinks=true,linkcolor=blue,urlcolor=blue,citecolor=blue]{hyperref}

\begin{document}

\title[Coupling coefficients of $su_q(1,1)$ \& multivariate $q$-Racah polynomials]{Coupling coefficients of $su_q(1,1)$ \\ and multivariate $q$-Racah polynomials}

\author{Vincent X. Genest}
\address{Department of Mathematics, Massachusetts Institute of Technology, Cambridge, MA 02139, USA}
\ead{vxgenest@mit.edu}

\author{Plamen Iliev}
\address{School of Mathematics, Georgia Institute of Technology, Atlanta, GA 30332, USA}
\ead{iliev@math.gatech.edu}

\author{Luc Vinet}
\address{Centre de recherches math\'ematiques, Universit\'e de Montr\'eal, Montr\'eal, QC H3C 3J7, Canada}
\ead{luc.vinet@umontreal.ca}

\begin{abstract}
Gasper \& Rahman's multivariate $q$-Racah polynomials are shown to arise as connection coefficients between families of multivariate $q$-Hahn or $q$-Jacobi polynomials. The families of $q$-Hahn polynomials are constructed as nested Clebsch--Gordan coefficients for the positive-discrete series representations of the quantum algebra $su_q(1,1)$. This gives an interpretation of the multivariate $q$-Racah polynomials in terms of $3nj$ symbols. It is shown that the families of $q$-Hahn polynomials also arise in wavefunctions of $q$-deformed quantum Calogero--Gaudin superintegrable systems.\\

\end{abstract}

\section*{Introduction}

This paper shows that Gasper \& Rahman's multivariate $q$-Racah polynomials arise as the connection coefficients between two families of multivariate $q$-Hahn or $q$-Jacobi polynomials. The two families of $q$-Hahn polynomials are constructed as nested Clebsch--Gordan coefficients for the positive-discrete series representations of the quantum algebra $su_q(1,1)$. This  result gives an algebraic interpretation of the multivariate $q$-Racah polynomials as recoupling coefficients, or $3nj$-symbols, of $su_q(1,1)$. It is also shown that the families of $q$-Hahn polynomials arise in wavefunctions of $q$-deformed quantum Calogero--Gaudin superintegrable systems of arbitrary dimension.

The multivariate $q$-Racah polynomials considered in this paper were originally introduced by Gasper and Rahman in \cite{Gasper&Rahman_2007} as $q$-analogs of the multivariate Racah polynomials defined by Tratnik in \cite{Tratnik1991a, Tratnik1991}. These $q$-Racah polynomials sit at the top of a hierarchy of orthogonal polynomials that extends the Askey scheme of (univariate) $q$-orthogonal polynomials; Tratnik's Racah polynomials and their descendants similarly generalize the Askey scheme at $q=1$. These two hierarchies will be referred to as the Gasper--Rahman and Tratnik schemes, respectively. The Gasper--Rahman scheme of multivariate $q$-orthogonal polynomials should be distinguished from the other multivariate extension of the Askey scheme based on root systems, which includes the Macdonald--Koornwinder polynomials \cite{Koornwinder1992} and the $q$-Racah polynomials defined by van Diejen and Stokman \cite{Diejen1998}. 

Like the families of univariate polynomials from the Askey scheme, the polynomials of the Gasper--Rahman and the Tratnik schemes are bispectral. Indeed, as shown by Iliev in \cite{Iliev2011}, and by Geronimo and Iliev in \cite{Geronimo2010}, these polynomials simultaneously diagonalize a pair of commutative algebras of operators that act on the degrees and on the variables of the polynomials, respectively. The bispectral property is a key element in the link between these families of polynomials, superintegrable systems, recoupling of algebra representations, and connection coefficients of multivariate orthogonal polynomials. Recall that a quantum system with $d$ degrees of freedom governed by a Hamiltonian $H$ is deemed maximally superintegrable if it admits $2d-1$ algebraically independent symmetry operators, including $H$ itself, that commute with the Hamiltonian \cite{Miller2013}. 

For the univariate Racah polynomials, one has the following picture \cite{Genest2014}. First, upon considering the 3-fold tensor product representations of $su(1,1)$, one finds that the two intermediate Casimir operators associated to adjacent pairs of representations in the tensor product satisfy the (rank one) Racah algebra, which is also the algebra generated by the two operators involved in the bispectral property of the univariate Racah polynomials. This leads to the identification of the Racah polynomials as $6j$ (or Racah) coefficients of $su(1,1)$, which are the transition coefficients between the two eigenbases corresponding to the diagonalization of the intermediate Casimir operators. Second, if one chooses the three representations being tensored to belong to the positive-discrete series, the total Casimir operator for the 3-fold tensor product representation can be identified with the Hamiltonian of the so-called generic superintegrable system on the 2-sphere, and the intermediate Casimir operators correspond to its symmetries. Finally, one obtains the interpretation of the univariate Racah polynomials as connection coefficients between two families of 2-variable Jacobi polynomials that arise as wavefunctions of the superintegrable Hamiltonian. For a review of the connection between the Askey scheme and superintegrable systems, see \cite{Kalnins2013}. For a review of the approach just described, the reader can also consult \cite{Genest2014b, Genest2015}.

The picture described above involving the one-variable Racah polynomials has recently been fully generalized to Tratnik's multivariate Racah polynomials. In \cite{Iliev2015}, Iliev and Xu have shown that these polynomials arise as connection coefficients between bases of multivariate Jacobi polynomials on the simplex \cite{DX} and used this to compute connection coefficients for families of discrete classical orthogonal polynomials studied in \cite{IX2007}, as well as for orthogonal polynomials on balls and spheres. In \cite{Iliev2016}, Iliev has established the connection between the bispectral operators for the multivariate Racah polynomials and the symmetries of the generic superintegrable system on the $d$-sphere. In \cite{DeBie2016}, De Bie, Genest, van de Vijver and Vinet have unveiled the relationship between this superintegrable model, $d$-fold tensor product representations of $su(1,1)$ and the higher rank Racah algebra. See also \cite{Genest2014a, Post2015, Rosengren1999}.

Many of these results have yet to be extended to the $q$-deformed case. The interpretation of the univariate $q$-Racah polynomials as $6j$ coefficients for the quantum algebra $su_q(1,1)$ is well known \cite{Vilenkin1992}, as is its relation with the Zhedanov algebra and the operators involved in the bispectrality of the one-variable $q$-Racah polynomials \cite{Granovskii1993, Zhedanov1991}. Dunkl has also shown  in \cite{Dunkl1980} that these polynomials arise as connection coefficients between bases of two-variable $q$-Jacobi or $q$-Hahn polynomials. On the multivariate side, Rosengren has observed that $q$-Hahn polynomials arise by considering nested Clebsch--Gordan coefficients for $su_q(1,1)$  and derived some explicit formulas \cite{Rosengren2001}. Moreover, Scarabotti has examined similar families of multivariate $q$-Hahn polynomials associated to binary trees and their connection coefficients \cite{Scarabotti2011}. 

Nevertheless, the identification of the multivariate $q$-Racah polynomials as $3nj$ coefficients of the quantum algebra $su_q(1,1)$ has not been achieved. Moreover, the connection between $q$-Racah polynomials, both univariate and multivariate, and superintegrable systems remains to be determined. The present paper addresses these questions. As stated above, it will be shown that Gasper \& Rahman's multivariate $q$-Racah polynomials arise as connection coefficients between bases of multivariate orthogonal $q$-Hahn or $q$-Jacobi polynomials. The bases will be constructed using the nested Clebsch--Gordan coefficients for multifold tensor product representations of $su_q(1,1)$, which will provide the exact interpretation of the multivariate $q$-Racah polynomials in terms of coupling coefficients for that quantum algebra. Finally, we will indicate how these bases also serve as eigenbases for $q$-deformed Calogero--Gaudin superintegrable systems.

The paper is organized as follows. In Section \ref{se1}, background material on $su_q(1,1)$ and its positive-discrete representations is provided. In Section \ref{se2}, the generalized Clebsch--Gordan problem of $su_q(1,1)$ is considered. The bases of multivariate $q$-Hahn polynomials are introduced and their bispectrality is related to commutative subalgebras of $su_q(1,1)^{\otimes d}$. In Section \ref{se3}, it is shown that the multivariate $q$-Racah polynomials arise as connecting coefficients between these bases. The proof of the one variable case relies on a new generating function argument. In Section \ref{se4}, the connection with superintegrability is established. We conclude with an outlook. 

\section{Basics of $su_q(1,1)$}\label{se1}
This section provides the necessary background material on the quantum algebra $su_q(1,1)$. In particular, the coproduct and the intermediate Casimir operators are introduced, and the representations of the positive-discrete series are defined.

\subsection{$su_q(1,1)$, tensor products and Casimir operators}
Let $q$ be a real number such that $0 < q < 1$. The quantum algebra $su_q(1,1)$ has three generators $A_0$, $A_{\pm}$ that satisfy the defining relations
\begin{align}
\label{SUq-Relations}
A_{-}A_{+} - q A_{+} A_{-} = \frac{q^{2A_0}-1}{q^{1/2}-q^{-1/2}},\qquad [A_{0}, A_{\pm}] = \pm A_{\pm},
\end{align}
where $[A,B] = AB-BA$. Upon taking $\widetilde{A}_{+} = A_{+} q^{-A_0/2}$, $\widetilde{A}_{-} = q^{-A_0/2} A_{-}$ and $\widetilde{A}_0 = A_0$, one recovers the defining relations of $su_q(1,1)$ in their usual presentation, that is
\begin{align*}
[\widetilde{A}_{-}, \widetilde{A}_{+}] =\frac{q^{\widetilde{A}_0}-q^{-\widetilde{A}_0}}{q^{1/2}-q^{-1/2}}, \qquad 
[\widetilde{A}_0,\widetilde{A}_{\pm}] = \pm \widetilde{A}_{\pm}.
\end{align*}
The Casimir operator $\Gamma$, which commutes with all generators, has the expression
\begin{align}
\label{SUq-Casimir}
\Gamma = \frac{q^{-1/2} q^{A_0} + q^{1/2} q^{-A_0}}{(q^{1/2}-q^{-1/2})^2} - A_{+} A_{-} q^{1-A_0}.
\end{align}
The coproduct map $\Delta : su_q(1,1) \rightarrow su_q(1,1) \otimes su_q(1,1)$ is defined as
\begin{align*}
\Delta (A_0) = A_0 \otimes 1 + 1 \otimes A_0,\qquad \Delta(A_{\pm}) = A_{\pm} \otimes 1 + q^{A_0}\otimes A_{\pm}.
\end{align*}
The coproduct $\Delta$ can be iterated to obtain embeddings of $su_q(1,1)$ into higher tensor powers. For a positive integer $d$, let $\Delta^{(d)}: su_q(1,1)\rightarrow su_q(1,1)^{\otimes d}$ be defined by
\begin{align}
\label{Coproduct}
\Delta^{(d)}= (1^{\otimes (d-2)}\otimes \Delta) \circ \Delta^{(d-1)}, \quad \Delta^{(1)} = \mathrm{Id},
\end{align}
where $\mathrm{Id}$ stands for the identity; one has $\Delta^{(2)} = \Delta$. 

For $1\leq i < j \leq d$, let $[i;j]$ denote the set $\{i, i+1, \ldots, j\}$. To each set $[i;j]$, one can associate a realization of $su_q(1,1)$ within $su_q(1,1)^{\otimes d}$. Denoting by $A_0^{(k)}$, $A_{\pm}^{(k)}$ the generators of the $k$\textsuperscript{th} factor of $su_q(1,1)$ in $su_q(1,1)^{\otimes d}$, the realization associated to the set $S=[i;j]$ has for generators
\begin{align}
\label{Gen-Rep}
A_0^{S} = \sum_{k=i}^{j}A_0^{(k)}, \qquad A_{\pm}^{S} = \sum_{k=i}^{j} q^{\sum_{\ell = i}^{k-1}A_0^{(\ell)}}A_{\pm}^{(k)}.
\end{align}
To each set $S=[i;j]$, one can thus associate an \emph{intermediate} Casimir operator $\Gamma^{S}$ defined as
\begin{align}
\label{Intermediate-Casimir}
\Gamma^{S} = \frac{q^{-1/2}q^{A_0^S} + q^{1/2}q^{-A_0^S}}{(q^{1/2}-q^{-1/2})^2} - A_{+}^S A_{-}^S q^{1-A_0^S}.
\end{align}
For a given value of $d$, the Casimir operator $\Gamma^{[1;d]}$ will be referred to as the \emph{full} Casimir operator. 
\subsection{Representations of the positive-discrete series}
Let $\alpha>0$ be a positive real number and let $V^{(\alpha)}$ be an infinite-dimensional vector space with orthonormal basis $e_{n}^{(\alpha)}$, where $n$ is a non-negative integer. The space $V^{(\alpha)}$ supports an irreducible representation of $su_q(1,1)$ defined by the actions
\begin{align}
\label{Actions}
A_0 e_{n}^{(\alpha)} = (n + (\alpha + 1)/2) e_{n}^{(\alpha)},\quad A_{+}e_{n}^{(\alpha)} = \sqrt{\sigma_{n+1}^{(\alpha)}} e_{n+1}^{(\alpha)},\qquad A_{-}e_{n}^{(\alpha)} = \sqrt{\sigma_{n}^{(\alpha)}} e_{n-1}^{(\alpha)},
\end{align}
with $\langle e_i, e_j\rangle = \delta_{ij}$ and where $\sigma_n$ is given by
\begin{align}
\label{Matrix-Elements}
\sigma_n^{(\alpha)} = q^{1/2}\frac{(1-q^n)(1-q^{n+\alpha})}{(1-q)^{2}}.
\end{align}
Note that $\sigma_0^{(\alpha)} = 0$ and that when $0 < q <1$, one has $\sigma_n^{(\alpha)} > 0$ for $n\geq 1$. It follows that $A_{\pm}^{\dagger}=A_{\mp}$, $A_0^{\dagger}=A_{0}$ as well as $\Gamma^{\dagger} = \Gamma$ on $V^{(\alpha)}$. The $su_q(1,1)$-modules $V^{(\alpha)}$ belong to the positive-discrete series. On $V^{(\alpha)}$, the Casimir operator \eqref{SUq-Casimir} acts as a multiple of the identity. Indeed, one easily verifies using \eqref{SUq-Casimir} and \eqref{Actions} that
\begin{align}
\label{Gamma}
\Gamma e_{n}^{(\alpha)} = \gamma(\alpha)\,e_{n}^{(\alpha)},\qquad \gamma(\alpha) = \frac{q^{\alpha/2} + q^{-\alpha/2}}{(q^{1/2}-q^{-1/2})^2}.
\end{align}
With the help of the nested coproduct defined in \eqref{Coproduct}, one can define tensor product representations of $su_q(1,1)$. Let $\bm{\alpha} = (\alpha_1, \alpha_2,\ldots, \alpha_{d})$ with $\alpha_i > 0$ be a $d$-dimensional multi-index and let $W^{(\bm{\alpha})}$ be the $d$-fold tensor product
\begin{align}
\label{d-Fold}
W^{(\bm{\alpha})} = V^{(\alpha_1)} \otimes \cdots \otimes  V^{(\alpha_d)}.
\end{align}
As per \eqref{Gen-Rep}, the space $W^{(\bm{\alpha})}$ supports a representation of $su_q(1,1)$ realized with the generators $A_0^{[1;d]}$, $A_{\pm}^{[1;d]}$. The space has a basis $e_{\bm{y}}^{(\bm{\alpha})}$ defined by
\begin{align}
\label{Direct-Product-Basis}
e_{\bm{y}}^{(\bm{\alpha})} = e_{y_1}^{(\alpha_1)}\otimes \cdots \otimes e_{y_{d}}^{(\alpha_d)},\qquad \langle e_{\bm{y}}^{(\bm{\alpha})}, e_{\bm{y}'}^{(\bm{\alpha})} \rangle = \delta_{\bm{y}\bm{y}'},
\end{align}
where $\bm{y} = (y_1,\ldots, y_{d})$ is a multi-index of non-negative integers. The action of the generators $A_0^{[1;d]}$, $A_{\pm}^{[1;d]}$ on the basis vectors \eqref{Direct-Product-Basis} is easily obtained by combining \eqref{Gen-Rep} and \eqref{Actions}. As a representation space, $W^{(\bm{\alpha})}$ is reducible and has the following decomposition in irreducible components:
\begin{align}
\label{Decompo}
W^{(\bm{\alpha})} = \bigoplus_{k = 0}^{\infty} m_k V^{(2k + A_d + d -1)},\qquad A_k = \sum_{i=1}^{k}\alpha_i, \qquad m_k = \binom{k+d-2}{k}.
\end{align}
When $d=2$, a detailed proof of the decomposition~\eqref{Decompo} can be found in \cite{Shibukawa1992}, see Theorem 2.1. The proof in arbitrary dimension follows easily by induction on $d$, using the identity
$$\sum_{k=0}^{s}\binom{k+a}{k}=\binom{s+a+1}{s}.$$
It follows from \eqref{Decompo} that the total Casimir operator $\Gamma^{[1;d]}$ defined by \eqref{Intermediate-Casimir} is diagonalizable on $W^{(\bm{\alpha})}$. Its eigenvalues $\lambda^{[1;d]}_{k}$ are given by
\begin{align*}
\lambda^{[1;d]}_{k} = \gamma(2k + A_d + d -1),\qquad k = 0,1,2,\ldots,
\end{align*}
where $\gamma(\alpha)$ is given by \eqref{Gamma}; these eigenvalues have multiplicity $m_k$.
\section{Multivariate $q$-Hahn bases \\ and the Clebsch--Gordan problem}\label{se2}
In this section, two orthogonal bases of multivariate $q$-Hahn polynomials are constructed in the framework of the generalized Clebsch--Gordan problem for $su_q(1,1)$. The connection between the standard Clebsch--Gordan problem involving two-fold tensor product representations and the one-variable $q$-Hahn polynomials is first reviewed, and then generalized to the multifold tensor product case. Two related bases of multivariate $q$-Jacobi polynomials are also introduced through a limit.
\subsection{Univariate $q$-Hahn polynomials and Clebsch--Gordan coefficients}
We first consider the $d=2$ case where $W^{(\alpha_1,\alpha_2)} = V^{(\alpha_1)} \otimes V^{(\alpha_2)}$. In addition to the direct product basis \eqref{Direct-Product-Basis}, the space $W^{(\alpha_1,\alpha_2)}$ admits another basis which is associated to its multiplicity-free decomposition \eqref{Decompo} in irreducible components. These basis elements $f_{n_1,n_2}^{(\alpha_1,\alpha_2)}$ are defined by the eigenvalue equations
\begin{align*}
\Gamma^{[1;2]} f_{n_1,n_2}^{(\alpha_1,\alpha_2)} = \gamma(2n_1 + A_2 +1)f_{n_1,n_2}^{(\alpha_1,\alpha_2)},\quad q^{A_0^{[1;2]}} f_{n_1,n_2}^{(\alpha_1,\alpha_2)} = q^{n_1 + n_2 +(A_2 + 2)/2} f_{n_1,n_2}^{(\alpha_1,\alpha_2)},
\end{align*}
with $n_1, n_2$ non-negative integers and where $A_2 = \alpha_1 + \alpha_2$.
\begin{Remark}
\label{rem0}
In the expansion $V^{(\alpha_1, \alpha_2)} = \bigoplus_{k=0}^{\infty} V^{(2k + \alpha_1 + \alpha_2 +1)}$, the ``coupled'' basis vector $f_{n_1,n_2}^{(\alpha_1,\alpha_2)} $ corresponds to $e_{n_2}^{(2n_1 + \alpha_1 + \alpha_2 +1)}\in V^{(2n_1 + \alpha_1 + \alpha_2 +1)}$.
\end{Remark}
The Clebsch--Gordan coefficients are the (real) expansion coefficients between the coupled basis $f_{n_1,n_2}^{(\alpha_1,\alpha_2)}$ defined above and the direct product basis $e_{y_1,y_2}^{(\alpha_1,\alpha_2)} = e_{y_1}^{(\alpha_1)}\otimes e_{y_2}^{(\alpha_2)}$. One has
\begin{align}
\begin{aligned}
\label{2-Clebsch}
f_{n_1,n_2}^{(\alpha_1,\alpha_2)} = \sum_{y_1, y_2} 
C_{n_1,n_2}^{(\alpha_1,\alpha_2)}(y_1,y_2) \; e_{y_1,y_2}^{(\alpha_1,\alpha_2)},
\qquad 
e_{y_1,y_2}^{(\alpha_1,\alpha_2)} = \sum_{n_1,n_2} 
C_{n_1,n_2}^{(\alpha_1,\alpha_2)}(y_1,y_2) \; f_{n_1,n_2}^{(\alpha_1,\alpha_2)},
\end{aligned}
\end{align}
where $C_{n_1,n_2}^{(\alpha_1,\alpha_2)}(y_1,y_2) = \langle f_{n_1,n_2}^{(\alpha_1,\alpha_2)}, e_{y_1,y_2}^{(\alpha_1,\alpha_2)}\rangle =
\langle e_{y_1,y_2}^{(\alpha_1,\alpha_2)}, f_{n_1,n_2}^{(\alpha_1,\alpha_2)} \rangle$ are the Clebsch--Gordan coefficients. These coefficients vanish unless $n_1 + n_2 = y_1 + y_2$. They enjoy the explicit expression \cite{Vilenkin1992}
\begin{align}
\label{Clebsch-q-Hahn}
C_{n_1,n_2}^{(\alpha_1,\alpha_2)}(y_1,y_2) = \delta_{n_1+n_2, y_1+y_2} \; \widehat{h}_{n_1}(y_1, \alpha_1, \alpha_2, y_1 + y_2;q),
\end{align}
where $\widehat{h}_{n}(x,\alpha,\beta, N;q)$ are the orthonormal $q$-Hahn polynomials multiplied by the square root of their weight function. One has
\begin{align}
\label{H-Hat}
\widehat{h}_{n}(x,\alpha,\beta, N;q) = \sqrt{\frac{\omega(x;\alpha,\beta,N;q)}{\eta(n,\alpha,\beta,N;q)}} \; h_{n}(x,\alpha,\beta,N;q),
\end{align}
with $h_{n}(x,\alpha,\beta,N;q)$ the $q$-Hahn polynomials \cite{Koekoek2010}
\begin{align}
\label{q-Hahn}
h_{n}(x,\alpha,\beta,N;q) = (q^{\alpha+1};q)_{n} (q^{-N};q)_{n}\;{}_3\phi_2\left(\genfrac{}{}{0pt}{}{q^{-n}, q^{n+\alpha+\beta+1}, q^{-x}}{q^{\alpha+1}, q^{-N}}; q, q\right),
\end{align}
where $(a;q)_{n}$ is the $q$-Pochhammer symbol
\begin{align*}
(a;q)_{n} = (1-a) (1-aq)\cdots (1-a q^{n-1}),\qquad (a;q)_{0} = 1,
\end{align*}
for $n$ a non-negative integer, and where ${}_r\phi_s$ is the basic hypergeometric series \cite{Gasper2004}
\begin{align*}
{}_r\phi_{s}\left(\genfrac{}{}{0pt}{}{a_1, \ldots, a_{r}}{b_1,\ldots, b_{s}}; q, z\right)
=
\sum_{k=0}^{\infty}
\frac{(a_1;q)_{k}\cdots (a_r;q)_{k}}{(q;q)_{k}(b_1;q)_{k}\cdots (b_s;q)_{k}} \left[(-1)^{k}q^{\binom{k}{2}}\right]^{1+s-r} z^{k}.
\end{align*}
The weight function $\omega(x;\alpha,\beta,N;q)$ and the normalization coefficient $\eta(n,\alpha,\beta,N;q)$ have the expressions
\begin{align*}
\omega(x;\alpha,\beta,N;q) = \frac{(q^{\alpha+1}; q)_{x} (q^{\beta+1}; q)_{N-x}}{(q;q)_{x}(q;q)_{N-x}}q^{(N-x)(\alpha+1)},
\end{align*}
and
\begin{multline*}
\eta(n,\alpha,\beta,N;q) = \frac{ (q^{\alpha+\beta+2};q)_{n+N}}{(q;q)_{N-n}} q^{2\binom{n}{2}-2Nn} 
\\ \times
\frac{1-q^{\alpha+\beta+1}}{1-q^{2n+\alpha+\beta+1}} 
\frac{(q;q)_n (q^{\alpha+1};q)_n (q^{\beta+1};q)_n}{(q^{\alpha+\beta+1};q)_n}\; q^{n(\alpha+1)} .
\end{multline*}
The $q$-Hahn functions $\widehat{h}_{n}(x,\alpha,\beta, N;q)$ satisfy the orthogonality relation
\begin{align}
\label{Ortho-h}
\sum_{x=0}^{N} \widehat{h}_{n}(x,\alpha,\beta, N;q) \widehat{h}_{n'}(x,\alpha,\beta, N;q) = \delta_{nn'}.
\end{align}
The coefficients $C_{n_1,n_2}^{(\alpha_1,\alpha_2)}(y_1,y_2)$ satisfy the orthogonality relations
\begin{align}
\label{Ortho-C}
\begin{aligned}
\sum_{y_1, y_2} C_{n_1,n_2}^{(\alpha_1,\alpha_2)}(y_1,y_2)\, C_{n_1',n_2'}^{(\alpha_1,\alpha_2)}(y_1,y_2) &= \delta_{n_1 n_1'}\delta_{n_2 n_2'},
\\
\sum_{n_1,n_2} C_{n_1,n_2}^{(\alpha_1,\alpha_2)}(y_1,y_2)\, C_{n_1,n_2}^{(\alpha_1,\alpha_2)}(y_1',y_2') &= \delta_{y_1 y_1'}\delta_{y_2 y_2'}.
\end{aligned}
\end{align}
If we define the shift operators $T_{y_i}^{\pm} f(y_i)=f(y_i \pm 1)$, then using  \eqref{Actions} we see that the generators of each copy of $su_q(1,1)$ will be represented by the following operators
\begin{align}
\label{Other-Realization}
q^{A_0^{(i)}} \mapsto q^{y_i + (\alpha_i+1)/2},\quad A_{+}^{(i)} \mapsto \sqrt{\sigma_{y_i+1}^{(\alpha_i)}} T_{y_i}^{+},\quad A_{-}^{(i)} \mapsto \sqrt{\sigma_{y_i}^{(\alpha_i)}} T_{y_i}^{-}.
\end{align}
In particular, per \eqref{Intermediate-Casimir}, the Casimir $\Gamma^{[1;2]}$ is identified with the operator
\begin{multline*}
\Gamma^{[1;2]} \mapsto \left(\gamma(\alpha_1)q^{-(y_2+(\alpha_2+1)/2)} +\gamma(\alpha_2) q^{y_1+(\alpha_1+1)/2} -\left(\frac{q^{1/2} + q^{-1/2}}{(q^{1/2}-q^{-1/2})^2}\right) q^{y_1-y_2+(\alpha_1-\alpha_2)/2}\right)
\\
- q^{-(y_2+(\alpha_2-1)/2)}\sqrt{\sigma_{y_1+1}^{(\alpha_1)}\sigma_{y_2}^{(\alpha_2)}} T_{y_1}^{+}T_{y_2}^{-} - q^{-(y_2+(\alpha_2+1)/2)} \sqrt{\sigma_{y_1}^{(\alpha_1)}\sigma_{y_2+1}^{(\alpha_2)}}T_{y_1}^{-}T_{y_2}^{+}.
\end{multline*}
By construction, this operator acts in a diagonal fashion on the functions $C_{n_1,n_2}^{(\alpha_1,\alpha_2)}(y_1,y_2)$ with eigenvalues $\gamma(2n_1 + A_2 +1)$. 

\begin{Remark}
\label{Rem1}
Let us note that the functions $C_{n_1,n_2}^{(\alpha_1,\alpha_2)}(y_1,y_2)$ are bispectral. Indeed, if we consider $\bm{y}$ and $\bm{n}$ such that 
$$y_1+y_2=n_1+n_2=N,$$
then in view of \eqref{Clebsch-q-Hahn}, we can think of $C_{n_1,n_2}^{(\alpha_1,\alpha_2)}(y_1,y_2)$ as an orthonormal $q$-Hahn polynomial of the variable $y_1$, with index $n_1$, multiplied by the square root of the weight. In addition to the spectral equation in $y_1$ stemming from the realization \eqref{Other-Realization}, they obey also a recurrence relation in the index $n_1$. This property, which is well known (see \cite[Section 14.6]{Koekoek2010}), can be derived explicitly in the present context by considering the matrix element 
$$\langle q^{A_0^{(1)}} f_{n_1,n_2}^{(\alpha_1,\alpha_2)}, e_{y_1,y_2}^{(\alpha_1,\alpha_2)}\rangle = \langle f_{n_1,n_2}^{(\alpha_1,\alpha_2)}, q^{A_0^{(1)}}e_{y_1,y_2}^{(\alpha_1,\alpha_2)} \rangle = q^{y_1 +(\alpha_1+1)/2} C_{n_1,n_2}^{(\alpha_1,\alpha_2)}(y_1,y_2),$$ 
and by computing the action of $q^{A_0^{(1)}} $ on $f_{n_1,n_2}^{(\alpha_1,\alpha_2)}$.
\end{Remark}

\subsection{Nested Clebsch--Gordan coefficients and $q$-Hahn bases}
We now consider the general $d$-fold tensor product representation $W^{(\bm{\alpha})}$ defined in \eqref{d-Fold}, and its decomposition in irreducible components \eqref{Decompo}. Since there are multiplicities in the decomposition, the eigenvalue problem for the total Casimir operator $\Gamma^{[1;d]}$ is degenerate. In the following, we will construct two bases associated to the diagonalization of two different sequences of intermediate Casimir operators.
\subsubsection{First basis} 
Consider the following sequence of intermediate Casimir operators: $\{\Gamma^{[1;2]}, \Gamma^{[1;3]}, \ldots, \Gamma^{[1;d-1]}, \Gamma^{[1;d]}\}$. By construction, these operators commute with one another, and can thus be diagonalized simultaneously. Moreover, since the intermediate Casimir operators are self-adjoint on $W^{(\bm{\alpha})}$, the resulting basis will be orthogonal. Taking $\bm{n}=(n_1,\ldots, n_{d})$, we define the orthonormal basis $g_{\bm{n}}^{(\bm{\alpha})}$ of $W^{(\bm{\alpha})}$ by the eigenvalue equations
\begin{align}
\label{Eigen}
\begin{aligned}
&\Gamma^{[1;k]} g_{\bm{n}}^{(\bm{\alpha})} =\gamma(2 N_{k-1} + A_k + k - 1) g_{\bm{n}}^{(\bm{\alpha})},
\qquad k=2,\ldots, d,
\\
&q^{A_0^{[1;d]}}g_{\bm{n}}^{(\bm{\alpha})} = q^{N_d+(A_d+d)/2}g_{\bm{n}}^{(\bm{\alpha})},
\end{aligned}
\end{align}
where $N_{k}$ and $A_{k}$ are defined as $N_{k} = \sum_{i=1}^{k} n_i$ and $A_k = \sum_{i=1}^{k} \alpha_i$. We shall consider the expansion coefficients of the coupled basis $g_{\bm{n}}^{(\bm{\alpha})}$ in the direct product basis $e_{\mathbf{y}}^{(\bm{\alpha})}$. Upon iterating the Clebsch--Gordan decomposition \eqref{2-Clebsch}, one obtains the following result.
\begin{Proposition}
\label{Prop-1}
Let $\Psi_{\bm{n}}^{(\bm{\alpha})}(\bm{y})$ be defined by the expression
\begin{align}
\label{Psi}
\Psi_{\bm{n}}^{(\bm{\alpha})}(\bm{y}) = \delta_{N_d Y_d}\prod_{k=1}^{d-1} \widehat{h}_{n_{k}}(Y_{k}-N_{k-1}, 2 N_{k-1} + A_{k} + k-1, \alpha_{k+1}, Y_{k+1}-N_{k-1}; q),
\end{align}
where $Y_{k}= \sum_{i=1}^{k} y_{i}$, and where $\widehat{h}_{n}(x,\alpha,\beta, N;q)$ is given by \eqref{H-Hat}. These functions satisfy the orthogonality relations
\begin{align}
\label{Ortho-Psi}
\sum_{\bm{y}} \Psi_{\bm{n}}^{(\bm{\alpha})}(\bm{y}) \Psi_{\bm{n}'}^{(\bm{\alpha})}(\bm{y})=\delta_{\bm{n}\bm{n}'},
\qquad
\sum_{\bm{n}} \Psi_{\bm{n}}^{(\bm{\alpha})}(\bm{y}) \Psi_{\bm{n}}^{(\bm{\alpha})}(\bm{y}')=\delta_{\bm{y}\bm{y}'},
\end{align}
and arise in the expansion formulas
\begin{align}
\label{Exp-2}
g_{\bm{n}}^{(\bm{\alpha})} = \sum_{\bm{y}}\Psi_{\bm{n}}^{(\bm{\alpha})}(\bm{y})\, e_{\mathbf{y}}^{(\bm{\alpha})}, 
\qquad 
e_{\mathbf{y}}^{(\bm{\alpha})} = \sum_{\bm{n}}\Psi_{\bm{n}}^{(\bm{\alpha})}(\bm{y})\, g_{\bm{n}}^{(\bm{\alpha})}.
\end{align}
\end{Proposition}
\begin{proof}
One starts from the direct product basis $e_{\mathbf{y}}^{(\bm{\alpha})}$ of $W^{(\bm{\alpha})}$. One can diagonalize the intermediate Casimir operator $\Gamma^{[1;2]}$ using the expansion \eqref{2-Clebsch}. This leads to
\begin{align*}
e_{\mathbf{y}}^{(\bm{\alpha})} = \sum_{n_1, \widehat{n}_1} C_{n_1,\widehat{n}_1}^{(\alpha_1,\alpha_2)}(y_1, y_2) \; f_{n_1, \widehat{n}_1}^{(\alpha_1, \alpha_2)} \otimes e_{y_3}^{(\alpha_3)}\otimes \cdots e_{y_{d}}^{(\alpha_d)}.
\end{align*}
Upon using Remark \ref{rem0} to identify $f_{n_1, \widehat{n}_1}^{(\alpha_1, \alpha_2)}$ with $e_{\widehat{n}_1}^{(2n_1 + \alpha_1 + \alpha_2 +1)}\in V^{(\alpha_1)}\otimes V^{(\alpha_2)}$, one can use \eqref{2-Clebsch} on $f_{n_1, \widehat{n}_1}^{(\alpha_1, \alpha_2)} \otimes e_{y_3}^{(\alpha_3)}$ to diagonalize $\Gamma^{[1;2]}$ and  $\Gamma^{[1;3]}$ simultaneously. One then obtains
\begin{multline*}
e_{\mathbf{y}}^{(\bm{\alpha})} = \sum_{n_1, \widehat{n}_1,n_2, \widehat{n}_2} 
C_{n_1,\widehat{n}_1}^{(\alpha_1,\alpha_2)}(y_1, y_2) \; C_{n_2,\widehat{n}_2}^{(2n_1+\alpha_1+\alpha_2+1,\alpha_3)}(\widehat{n}_1, y_3)
\\\times
f_{n_2, \widehat{n}_2}^{(2n_1 +  \alpha_1 +\alpha_2 +1,\alpha_3)}\otimes e_{y_4}^{(\alpha_4)}\otimes \cdots \otimes e_{y_{d}}^{(\alpha_d)}.
\end{multline*}
Using Remark \ref{rem0} again to identify $f_{n_2, \widehat{n}_2}^{(2n_1 +  \alpha_1 +\alpha_2 +1,\alpha_3)}$ with the vector $e_{\widehat{n}_2}^{(2n_1 + 2n_2 + \alpha_1 + \alpha_2 + \alpha_3 +2)} \in  V^{(\alpha_1)}\otimes V^{(\alpha_2)} \otimes V^{(\alpha_3)}$, one can use \eqref{2-Clebsch} on $f_{n_2, \widehat{n}_2}^{(2n_1 + \alpha_1 +\alpha_2 +1 , \alpha_3)} \otimes e_{y_4}^{(\alpha_4)}$ to diagonalize $\Gamma^{[1;2]}$, $\Gamma^{[1;3]}$ and $\Gamma^{[1;4]}$ simultaneously. This leads to
\begin{multline*}
 e_{\mathbf{y}}^{(\bm{\alpha})}  = \sum_{n_1, \hat{n}_1}  \sum_{n_2, \hat{n}_2}  \sum_{n_3, \hat{n}_3} 
C_{n_1,\hat{n}_1}^{(\alpha_1,\alpha_2)}(y_1, y_2) C_{n_2,\hat{n}_2}^{(2n_1+\alpha_1+\alpha_2+1,\alpha_3)}(\hat{n}_1, y_3)
\\
\times  C_{n_3,\hat{n}_3}^{(2n_1+2n_2+\alpha_1+\alpha_2+\alpha_3+2,\alpha_4)}(\hat{n}_2, y_4)\;\;
f_{n_3,\hat{n}_3}^{(2n_1+2n_2+\alpha_1+\alpha_2+\alpha_3+2,\alpha_4)}
\otimes e_{y_5}^{(\alpha_5)}\otimes \cdots \otimes e_{y_{d}}^{(\alpha_d)},
\end{multline*}
and so on. Then one can use \eqref{Clebsch-q-Hahn} to deduce that the terms in the sum above vanish unless
\begin{align*}
\widehat{n}_j = Y_{j+1} - N_{j},\qquad j = 1, 2, \ldots, d-1.
\end{align*}
Upon substituting the above condition in the expansion of $e_{\mathbf{y}}^{(\bm{\alpha})}$ we obtain the second formula in \eqref{Exp-2} where the coefficients $\Psi_{\bm{n}}^{(\bm{\alpha})}(\bm{y})$ are given in \eqref{Psi}. Since both bases $\{e_{\bm{y}}^{(\bm{\alpha})}\}$ and $\{g_{\bm{n}}^{(\bm{\alpha})}\}$ are orthonormal, \eqref{Ortho-Psi} and the first formula in \eqref{Exp-2} follow from the fact that the matrix $(\Psi_{\bm{n}}^{(\bm{\alpha})}(\bm{y}))_{\bm{n}, \bm{y}}$ is orthogonal.
\end{proof}
The coefficients $\Psi_{\bm{n}}^{(\bm{\alpha})}(\bm{y})$ can be viewed as nested Clebsch--Gordan coefficients for the positive discrete series of irreducible representations of $su_q(1,1)$. A different approach to obtain multivariate $q$-Hahn polynomials was outlined by Rosengren in \cite{Rosengren2001}. Note that 
$$\Psi_{\bm{n}}^{(\bm{\alpha})}(\bm{y}) = \langle g_{\bm{n}}^{(\bm{\alpha})},  e_{\mathbf{y}}^{(\bm{\alpha})}\rangle .$$
Using \eqref{Other-Realization}, we can represent the action of $\Gamma^{[1;k]}$ as a difference operator in the variables $\bm{y}$. Then, the equation
$$ \langle \Gamma^{[1;k]} g_{\bm{n}}^{(\bm{\alpha})},  e_{\mathbf{y}}^{(\bm{\alpha})}\rangle = \langle g_{\bm{n}}^{(\bm{\alpha})}, \Gamma^{[1;k]} e_{\mathbf{y}}^{(\bm{\alpha})}\rangle$$
combined with \eqref{Eigen} shows that $\Psi_{\bm{n}}^{(\bm{\alpha})}(\bm{y})$ are eigenfunctions of the difference operators  $\Gamma^{[1;k]}$ with eigenvalues $\gamma(2 N_{k-1} + A_k + k - 1)$. The results in \cite{Iliev2011} imply that the functions $\Psi_{\bm{n}}^{(\bm{\alpha})}(\bm{y})$ are also eigenfunctions of commuting operators acting on the variables $\bm{n}$, and therefore are bispectral. The natural extension of the arguments in Remark \ref{Rem1} provides a Lie-interpretation of the corresponding bispectral algebras of partial difference operators.

\subsubsection{A family of multivariate $q$-Hahn polynomials}
The orthonormal functions $\Psi_{\bm{n}}^{(\bm{\alpha})}(\bm{y})$ can be expressed in terms of the multivariate $q$-Hahn polynomials introduced by Gasper and Rahman in \cite{Gasper&Rahman_2007}. Indeed, one can write
\begin{align}
\label{Psi-Pols}
\Psi_{\bm{n}}^{(\bm{\alpha})}(\bm{y}) = \delta_{N_d Y_d}
\sqrt{\frac{\rho^{(\bm{\alpha})}(\bm{y})}{\Lambda_{\bm{n}}^{(\bm{\alpha})}}}\; H_{\bm{n}}^{(\bm{\alpha})}(\bm{y}),
\end{align}
with $\rho^{(\bm{\alpha})}(\bm{y})$ given by
\begin{align}
\label{Rho-Weight}
\rho^{(\bm{\alpha})}(\bm{y}) = \prod_{k=1}^{d}\frac{(q^{\alpha_k+1}; q)_{y_k}}{(q;q)_{y_k}} q^{y_k(A_{k-1} + k - 1)}.
\end{align}
The normalization factor $\Lambda_{\bm{n}}^{(\bm{\alpha})}$ has the expression
\begin{multline*}
\Lambda_{\bm{n}}^{(\bm{\alpha})} = \frac{(q^{A_d + d};q)_{n_d + 2N_{d-1}}}{(q;q)_{n_d}} q^{2\binom{N_{d-1}}{2} - 2 N_{d}N_{d-1}}
\\
\times 
\prod_{k=1}^{d-1}
\frac{1-q^{A_{k+1}+k}}{1-q^{2 N_k + A_{k+1} + k}} \frac{(q;q)_{n_k} (q^{A_k + k};q)_{n_k + 2 N_{k-1}}(q^{\alpha_{k+1}+1};q)_{n_k}}{(q^{A_{k+1} + k};q)_{n_k + 2 N_{k-1}}} q^{n_k(2N_{k-1}+A_k+k)},
\end{multline*}
and $H_{\bm{n}}^{(\bm{\alpha})}(\bm{y})$ are the Gasper--Rahman multivariate $q$-Hahn polynomials defined as
\begin{align}
\label{MV-q-Hahn}
H_{\bm{n}}^{(\bm{\alpha})}(\bm{y}) = \prod_{k=1}^{d-1}h_{n_k}(Y_{k}-N_{k-1}, 2 N_{k-1} + A_{k} + k -1, \alpha_{k+1}, Y_{k+1} - N_{k-1};q),
\end{align}
where $h_{n}(x,\alpha,\beta,N;q)$ are the $q$-Hahn polynomials given in \eqref{q-Hahn}. These multivariate $q$-Hahn polynomials are a direct $q$-deformation of Karlin \& McGregor's multivariate Hahn polynomials \cite{Karlin1975}. Upon fixing $M\in \mathbb{N}$ and taking $\bm{n}$ and $\bm{n}'$ such that $N_d=N_d'=M$, these polynomials satisfy the orthogonality relation
\begin{align}
\label{Ortho-H}
\sum_{\substack{\bm{y}\\ Y_d= M}} \rho^{(\bm{\alpha})}(\bm{y})\; H_{\bm{n}}^{(\bm{\alpha})}(\bm{y}) \;H_{\bm{n}'}^{(\bm{\alpha})}(\bm{y}) = \Lambda_{\bm{n}}^{(\bm{\alpha})}\delta_{\bm{n}\bm{n}'}.
\end{align}
\begin{Remark}
Note that when dealing with the polynomials \eqref{MV-q-Hahn}, one usually fixes $Y_{d} = N_{d} = M$ and takes $n_{d} = M - N_{d-1}$. The resulting polynomials have parameters $\alpha_1,\ldots, \alpha_{d}$ and $M$, and degree indices $n_1,\ldots,n_{d-1}$. Alternatively, $H_{\bm{n}}^{(\bm{\alpha})}(\bm{y})$ can be described as orthogonal polynomials of total degree $N_{d-1}$ in the variables $q^{-Y_1},q^{-Y_2}, \ldots, q^{-Y_{d-1}}$.
\end{Remark}
\subsubsection{Second basis} 
Consider the sequence of operators $\{\Gamma^{[2;3]}, \Gamma^{[2;4]}, \ldots, \Gamma^{[2;d]}, \Gamma^{[1;d]}\}$. These operators are self-adjoint on $W^{(\bm{\alpha})}$, and they commute with one another. Consequently, one can construct an orthonormal basis that simultaneously diagonalizes them. Taking $\bm{m}=(m_1,\ldots, m_{d})$, the orthonormal basis $u_{\bm{m}}^{(\bm{\alpha})}$ of $W^{(\bm{\alpha})}$ is defined by the eigenvalue equations
\begin{align}
\label{eigen-2}
\begin{aligned}
&\Gamma^{[2;k]} u_{\bm{m}}^{(\bm{\alpha})} =\gamma(2 M_{k-2} + \widetilde{A}_k + k - 2) u_{\bm{m}}^{(\bm{\alpha})},
\qquad
k=3,\ldots, d,
\\
&\Gamma^{[1;d]} u_{\bm{m}}^{(\bm{\alpha})} =\gamma(2 M_{d-1} + A_d + d - 1) u_{\bm{m}}^{(\bm{\alpha})}, \qquad
q^{A_0^{[1;d]}}u_{\bm{m}}^{(\bm{\alpha})} = q^{M_d+(A_d+d)/2} \, u_{\bm{m}}^{(\bm{\alpha})},
\end{aligned}
\end{align}
where  $M_k = \sum_{i=1}^{k} m_i$ and where $\widetilde{A}_k = \sum_{i=2}^{k} \alpha_i$. Once again we consider the expansion coefficients of the coupled basis $u_{\bm{m}}^{(\bm{\alpha})}$ in the direct product basis $e_{\bm{y}}^{(\bm{\alpha})}$. One has the following result.
\begin{Proposition}
Let $\Xi_{\bm{m}}^{(\bm{\alpha})}(\bm{y})$ be defined by the expression
\begin{multline}
\label{Xi-Def}
\Xi_{\bm{m}}^{(\bm{\alpha})}(\bm{y}) =\delta_{M_d Y_d} \prod_{k=1}^{d-2} \widehat{h}_{m_k}(\widetilde{Y}_{k+1}-M_{k-1}, 2M_{k-1}+\widetilde{A}_{k+1}+k-1, \alpha_{k+2}, \widetilde{Y}_{k+2}-M_{k-1};q)
\\
\times \widehat{h}_{m_{d-1}}(y_1, \alpha_1, 2 M_{d-2} + \widetilde{A}_{d} + d -2, Y_{d}- M_{d-2};q),
\end{multline}
where $\widetilde{Y}_{k} = \sum_{i=2}^{k} y_i$ and where $\widehat{h}_{n}(x,\alpha,\beta,N;q)$ is given by \eqref{H-Hat}. The function $\Xi_{\bm{m}}^{(\bm{\alpha})}(\bm{y})$ satisfy the orthogonality relations
\begin{align}
\label{Ortho-Xi}
\sum_{\bm{y}} \Xi_{\bm{m}}^{(\bm{\alpha})}(\bm{y}) \Xi_{\bm{m}'}^{(\bm{\alpha})}(\bm{y})=\delta_{\bm{m}\bm{m}'},
\qquad
\sum_{\bm{m}} \Xi_{\bm{m}}^{(\bm{\alpha})}(\bm{y}) \Xi_{\bm{m}}^{(\bm{\alpha})}(\bm{y}')=\delta_{\bm{y}\bm{y}'},
\end{align}
and arise in the expansion formulas
\begin{align}
\label{Exp-Xi}
u_{\bm{m}}^{(\bm{\alpha})} = \sum_{\bm{y}}\Xi_{\bm{m}}^{(\bm{\alpha})}(\bm{y})\, e_{\mathbf{y}}^{(\bm{\alpha})}, 
\qquad 
e_{\mathbf{y}}^{(\bm{\alpha})} = \sum_{\bm{m}}\Xi_{\bm{m}}^{(\bm{\alpha})}(\bm{y})\, u_{\bm{m}}^{(\bm{\alpha})}.
\end{align}
\end{Proposition}
\begin{proof}
The proof is along the same lines as that of Proposition \ref{Prop-1}. Starting from the direct product basis $e_{\bm{y}}^{(\bm{\alpha})}$, one diagonalizes the operator $\Gamma^{[2;3]}$ by using the expansion \eqref{2-Clebsch}. This leads to
\begin{align*}
e_{\bm{y}}^{(\bm{\alpha})} = \sum_{m_1, \widehat{m}_1} C_{m_1, \widehat{m}_1}^{(\alpha_2,\alpha_3)}(y_2, y_3)\; e_{y_1}^{(\alpha_1)}\otimes f_{m_1, \widehat{m}_1}^{(\alpha_2,\alpha_3)} \otimes e_{y_4}^{(\alpha_4)}\otimes \cdots \otimes e_{y_{d}}^{(\alpha_d)}.
\end{align*}
One can use Remark \ref{rem0} to identify $f_{m_1, \widehat{m}_1}^{(\alpha_2,\alpha_3)}$ with $e_{\widehat{m}_1}^{(2m_1 + \alpha_2 + \alpha_3 +1)} \in V^{(\alpha_2)}\otimes V^{(\alpha_3)}$ and use \eqref{2-Clebsch} on the vector $f_{m_1, \widehat{m}_1}^{(\alpha_2,\alpha_3)} \otimes e_{y_4}^{(\alpha_4)}$. Repeating this procedure until $\Gamma^{[2;d]}$ is diagonalized, one finally diagonalizes the total Casimir operator $\Gamma^{[1;d]}$ by applying \eqref{2-Clebsch} one last time. Using the explicit expression \eqref{Clebsch-q-Hahn} then yields \eqref{Xi-Def}.
\end{proof}
The coefficients $\Xi_{\bm{m}}^{(\bm{\alpha})}(\bm{y})$ can also be viewed as nested Clebsch--Gordan coefficients and satisfy bispectral equations.
\subsubsection{Another family of multivariate $q$-Hahn polynomials}
The orthogonal functions $\Xi_{\bm{m}}^{(\bm{\alpha})}(\bm{y})$ can also be written in terms of a family multivariate orthogonal polynomials of $q$-Hahn type. One has indeed
\begin{align}
\label{Xi-Pols}
\Xi_{\bm{m}}^{(\bm{\alpha})}(\bm{y}) = \delta_{M_{d} Y_{d}} \sqrt{\frac{\rho^{(\bm{\alpha})}(\bm{y})}{\Omega_{\bm{m}}^{(\bm{\alpha})}}}\;
G_{\bm{m}}^{(\bm{\alpha})}(\bm{y}),
\end{align}
where $\rho^{(\bm{\alpha})}(\bm{y})$ is given by \eqref{Rho-Weight}. The normalization factor  $\Omega_{\bm{m}}^{(\bm{\alpha})}$ is of the form
\begin{multline*}
\Omega_{\bm{m}}^{(\bm{\alpha})} = \frac{(q^{A_d+d};q)_{m_d + 2 M_{d-1}}q^{2\binom{M_{d-1}}{2}-2M_{d}M_{d-1}}}{(q;q)_{m_{d}}}
\\
\times \prod_{k=1}^{d-2} \frac{1-q^{\widetilde{A}_{k+2} + k }}{1-q^{2 M_k + \widetilde{A}_{k+2} + k}} 
\frac{(q;q)_{m_k} (q^{\widetilde{A}_{k+1} + k};q)_{m_k + 2 M_{k-1}} (q^{\alpha_{k+2}+ 1};q)_{m_k}}{(q^{\widetilde{A}_{k+2}+k};q)_{m_k + 2 M_{k-1}}} q^{m_k (2 M_{k-1} + \widetilde{A}_{k+1}+k)}
\\
\times \frac{1-q^{A_{d}+d-1}}{1-q^{2M_{d-1}+A_d+d-1}} 
\frac{(q;q)_{m_{d-1}} (q^{\widetilde{A}_d + d -1};q)_{m_{d-1} + 2M_{d-2}} (q^{\alpha_1+1};q)_{m_{d-1}} }{(q^{A_d + d -1};q)_{m_{d-1}+2M_{d-2}}}
q^{M_{d-1}(\alpha_1 + 1)},
\end{multline*}
and the polynomials $G_{\bm{m}}^{(\bm{\alpha})}(\bm{y})$ have the expression
\begin{multline}
\label{G-Pols}
 G_{\bm{m}}^{(\bm{\alpha})}(\bm{y})=
  \prod_{k=1}^{d-2} h_{m_k}(\widetilde{Y}_{k+1}-M_{k-1}, 2M_{k-1}+\widetilde{A}_{k+1}+k-1,\alpha_{k+2},\widetilde{Y}_{k+2} - M_{k-1};q)
\\
\times q^{-y_1 M_{d-2}} h_{m_{d-1}}(y_1,\alpha_1, 2M_{d-2}+\widetilde{A}_{d}+d-2,Y_{d} - M_{d-2};q),
\end{multline}
where $h_n(x,\alpha,\beta,N;q)$ are the $q$-Hahn polynomials \eqref{q-Hahn}. These polynomials are orthogonal with respect to the same measure as the Gasper--Rahman $q$-Hahn polynomials $H_{\bm{n}}^{(\bm{\alpha})}(\bm{y})$ given by \eqref{MV-q-Hahn}. Upon fixing $L\in \mathbb{N}$ and taking $\bm{m}$ and $\bm{m}'$ such that $M_d=M_d'=L$, the orthogonality relation for the polynomials $G_{\bm{m}}^{(\bm{\alpha})}(\bm{y})$ reads
\begin{align}
\label{Ortho-G}
\sum_{\substack{\bm{y}\\ Y_d = L}} \rho^{(\bm{\alpha})}(\bm{y})\; G_{\bm{m}}^{(\bm{\alpha})}(\bm{y}) \;G_{\bm{m}'}^{(\bm{\alpha})}(\bm{y}) = \Omega_{\bm{m}}^{(\bm{\alpha})}\delta_{\bm{m}\bm{m}'}.
\end{align}
\begin{Remark}
Once again, one can take $m_{d} = L - Y_{d-1}$. The resulting polynomials have parameters $\alpha_1,\ldots, \alpha_{d}$ and $L$, and degree indices $m_1,\ldots,m_{d-1}$. Alternatively, $G_{\bm{m}}^{(\bm{\alpha})}(\bm{y})$ can be described as orthogonal polynomials of total degree $M_{d-1}$ in the variables $q^{-Y_1},q^{-Y_2}, \ldots, q^{-Y_{d-1}}$.
\end{Remark}
Let us quickly recap the results obtained so far. We have used nested Clebsch--Gordan coefficients for multifold tensor product representations of $su_q(1,1)$ to construct two bases of multivariate orthogonal functions $\Psi_{\bm{n}}^{(\bm{\alpha})}(\bm{y})$ and $\Xi_{\bm{m}}^{(\bm{\alpha})}(\bm{y})$ that diagonalize two commutative subalgebras of intermediate Casimir operators. From $\Psi_{\bm{n}}^{(\bm{\alpha})}(\bm{y})$ and $\Xi_{\bm{m}}^{(\bm{\alpha})}(\bm{y})$, we then constructed two families of multivariate $q$-Hahn polynomials $H_{\bm{n}}^{(\bm{\alpha})}(\bm{y})$ and $G_{\bm{m}}^{(\bm{\alpha})}(\bm{y})$ that are orthogonal with respect to the same measure $\rho^{(\bm{\alpha})}(\bm{y})$ given in \eqref{Rho-Weight}.
\subsection{$q$-Jacobi bases}
The orthogonal functions $\Psi_{\bm{n}}^{(\bm{\alpha})}(\bm{y})$ and $\Xi_{\bm{m}}^{(\bm{\alpha})}(\bm{y})$ defined in \eqref{Psi} and \eqref{Xi-Def} are both non-zero if the condition $Y_{d} = N_{d} = M_{d} $ is satisfied. The corresponding families of multivariate $q$-Hahn polynomials $H_{\bm{n}}^{(\bm{\alpha})}(\bm{y})$ and $G_{\bm{m}}^{(\bm{\alpha})}(\bm{y})$ consequently satisfy the finite orthogonality relations \eqref{Ortho-H} and \eqref{Ortho-G}. It is therefore meaningful to consider limits of these families of functions and polynomials as $Y_d = N_d = M_d$ goes to infinity. We shall consider the limits of the functions $\Psi_{\bm{n}}^{(\bm{\alpha})}(\bm{y})$ and $\Xi_{\bm{m}}^{(\bm{\alpha})}(\bm{y})$ and extract orthogonal polynomials from each of those limits. These results will prove useful later.

Let $p_{n}(x,\alpha,\beta;q)$ be the little $q$-Jacobi polynomials \cite{Koekoek2010}
\begin{align}
\label{q-Jacobi}
p_{n}(x,\alpha,\beta;q) = (q^{\alpha+1};q)_{n} 
\;{}_2\phi_1\left(\genfrac{}{}{0pt}{}{q^{-n}, q^{n+\alpha+\beta+1}}{q^{\alpha+1}}; q, q^{x+1}\right).
\end{align}
Define the functions $\widehat{p}_{n}(x,\alpha,\beta;q)$ as follows:
\begin{align*}
\widehat{p}_{n}(x,\alpha,\beta;q) = 
\sqrt{\frac{\mu(x,\alpha,\beta;q)}{\kappa(n,\alpha,\beta;q)}}\;
p_n(x,\alpha,\beta;q),
\end{align*}
where $\mu(x,\alpha,\beta;q)$ and $\kappa(n,\alpha,\beta;q)$ are given by
\begin{align*}
\mu(x,\alpha,\beta;q) &= 
\frac{(q^{\alpha+1};q)_{\infty}}{(q^{\alpha + \beta +2};q)_{\infty}} 
\frac{(q^{\beta+1};q)_{x}}{(q;q)_x}q^{x(\alpha + 1)},
\\
\kappa(n,\alpha,\beta;q) &=
\frac{1-q^{\alpha + \beta + 1}}{1-q^{2n + \alpha + \beta + 1}}
\frac{(q;q)_{n} (q^{\alpha + 1};q)_{n} (q^{\beta+1};q)_{n}}{(q^{\alpha + \beta +1};q)_{n}} q^{n(\alpha+1)}.
\end{align*}
The functions  $\widehat{p}_{n}(x,\alpha,\beta;q)$ satisfy the orthogonality relation \cite{Koekoek2010}
\begin{align}
\label{Ortho-P}
\sum_{x \geq 0} \widehat{p}_{n}(x,\alpha,\beta;q) \; \widehat{p}_{m}(x,\alpha,\beta;q) = \delta_{nm}.
\end{align}
\subsubsection{A first basis of $q$-Jacobi functions} Let us first consider the limit of the functions $\Psi_{\bm{n}}^{(\bm{\alpha})}(\bm{y})$ as $Y_{d} = N_{d}$ goes to infinity. We find the following result.
\begin{Proposition}
Let $\bm{s} = (s_1,\ldots,s_{d-1})$,  $\bm{x} = (x_1,\ldots,x_{d-1})$ and $\bm{\alpha} = (\alpha_1,\ldots, \alpha_{d})$. Furthermore, let $\mathcal{J}_{\bm{s}}^{(\bm{\alpha})}(\bm{x})$ be the functions defined as 
\begin{align}
\label{J-Def}
\mathcal{J}_{\bm{s}}^{(\bm{\alpha})}(\bm{x}) = 
\prod_{k=1}^{d-1} 
(-1)^{s_k}\widehat{p}_{s_k}(x_k, 2 S_{k-1} + A_k + k -1, \alpha_{k+1} ; q),
\end{align}
where $S_{k} = \sum_{i=1}^{k} s_i$ and $A_k = \sum_{i=1}^{k} \alpha_i$. Upon taking $\widetilde{\bm{s}}_{L}=(s_1,s_2,\ldots, s_{d-1}, L - S_{d-1})$ and $\widetilde{\bm{x}}_{L} = (L -X_{d-1}, x_1, \ldots, x_{d-1})$, one has
\begin{align*}
\lim_{L\rightarrow \infty}  \Psi_{\widetilde{\bm{s}}_{L}}^{(\bm{\alpha})}(\widetilde{\bm{x}}_{L}) = 
\mathcal{J}_{\bm{s}}^{(\bm{\alpha})}(\bm{x}).
\end{align*}
The functions $\mathcal{J}_{\bm{s}}^{(\bm{\alpha})}(\bm{x})$ obey the orthogonality relation
\begin{align}
\label{Ortho-J}
\sum_{\bm{x}} \mathcal{J}_{\bm{s}}^{(\bm{\alpha})}(\bm{x})\; \mathcal{J}_{\bm{s}'}^{(\bm{\alpha})}(\bm{x}) = \delta_{\bm{s}\bm{s}'},
\end{align}
where the multi-index $\bm{x}$ runs over multi-indices of non-negative integers.
\end{Proposition}
\begin{proof}
The orthogonality relation \eqref{Ortho-J} follows from \eqref{Ortho-P}. The limit can be taken directly. The calculation is long, but otherwise straightforward. The following result relating the $q$-Hahn to the $q$-Jacobi polynomials is useful \cite{Gasper2004}:
\begin{align}
\label{Useful-Lim}
\lim_{N\rightarrow \infty} \frac{h_{n}(N-x,\alpha,\beta,N;q)}{(q^{-N};q)_{n}}  = p_n(x,\alpha,\beta,q).
\end{align}
\end{proof}
\subsubsection{A first basis of $q$-Jacobi polynomials}
The functions $\mathcal{J}_{\bm{s}}^{(\bm{\alpha})}(\bm{x})$ naturally give rise to a family of $q$-Jacobi polynomials with $d-1$ variables. Indeed, taking $\bm{x} = (x_1,\ldots, x_{d-1})$, $\bm{s} = (s_1,\ldots, s_{d-1})$ and $\bm{\alpha} = (\alpha_1,\ldots, \alpha_{d})$, one can write the functions $\mathcal{J}_{\bm{s}}^{(\bm{\alpha})}(\bm{x})$ as follows
\begin{align}
\label{J-Pols}
\mathcal{J}_{\bm{s}}^{(\bm{\alpha})}(\bm{x}) =  \sqrt{\frac{\nu^{(\bm{\alpha})}(\bm{x})}{\iota_{\bm{s}}^{(\bm{\alpha})}}}\;
J_{\bm{s}}^{(\bm{\alpha})}(\bm{x}),
\end{align}
with $\nu^{(\bm{\alpha})}(\bm{x})$ given by
\begin{align}
\label{Nu-Weight}
\nu^{(\bm{\alpha})}(\bm{x}) = 
\frac{(q^{\alpha_1 +1};q)_{\infty}}{(q^{A_d + d};q)_{\infty}}
\prod_{k=2}^{d} \frac{(q^{\alpha_k+1};q)_{x_{k-1}}}{(q;q)_{x_{k-1}}} q^{x_{k-1}(A_{k-1}+k-1)}.
\end{align}
The normalization factor $\iota_{\bm{s}}^{(\bm{\alpha})}$ has the expression
\begin{align*}
\iota_{\bm{s}}^{(\bm{\alpha})} = \prod_{k=1}^{d-1}
\frac{1-q^{A_{k+1} + k}}{1-q^{2S_k + A_{k+1} + k}}
\frac{(q;q)_{s_k}(q^{A_k + k};q)_{s_k + 2 S_{k-1}} (q^{\alpha_{k+1} + 1};q)_{s_k}}{(q^{A_{k+1} + k};q)_{s_k + 2 S_{k -1}}}
q^{s_k (2S_{k-1} + A_k + k)},
\end{align*}
and $J_{\bm{s}}^{(\bm{\alpha})}(\bm{x})$ are some multivariate $q$-Jacobi polynomials
\begin{align}
\label{MV-Jacobi}
J_{\bm{s}}^{(\bm{\alpha})}(\bm{x}) = \prod_{k=1}^{d-1} (-1)^{s_k}q^{x_k S_{k-1}}p_{s_k}(x_k, 2 S_{k-1} + A_k + k -1, \alpha_{k+1};q).
\end{align}
These polynomials satisfy the orthogonality relation
\begin{align*}
\sum_{\bm{x}} \nu^{(\bm{\alpha})}(\bm{x}) \; J_{\bm{s}}^{(\bm{\alpha})}(\bm{x})
\; J_{\bm{s}'}^{(\bm{\alpha})}(\bm{x}) = \iota_{\bm{s}}^{(\bm{\alpha})} \delta_{\bm{s} \bm{s}'},
\end{align*}
where $\bm{x}$ runs over the multi-indices of non-negative integers.
\begin{Remark}
The functions $J_{\bm{s}}^{(\bm{\alpha})}(\bm{x})$ are polynomials of total degree $S_{d-1}$ in the variables $q^{\sum_{j=1}^{d-1}x_j}, q^{\sum_{j=2}^{d-1}x_j}, \ldots, q^{x_{d-1}}$ with parameters $\alpha_1, \alpha_2, \ldots, \alpha_{d}$.
\end{Remark}
\subsubsection{A second basis of $q$-Jacobi functions} Let us now consider the limit of the functions $\Xi_{\bm{m}}^{(\bm{\alpha})}(\bm{y})$ as $M_{d} = Y_{d}$ goes to infinity. We obtain the following result.
\begin{Proposition}
Let $\bm{t} = (t_1,\ldots,t_{d-1})$,  $\bm{x} = (x_1,\ldots,x_{d-1})$ and $\bm{\alpha} = (\alpha_1,\ldots, \alpha_{d})$. Moreover, let $\mathcal{Q}_{\bm{t}}^{(\bm{\alpha})}(\bm{x})$ be the functions defined by
\begin{multline}
\label{Q-Def}
\mathcal{Q}_{\bm{t}}^{(\bm{\alpha})}(\bm{x}) =
\prod_{k=1}^{d-2} \widehat{h}_{t_k}(X_{k}-T_{k-1}, 2T_{k-1} + \widetilde{A}_{k+1} + k -1, \alpha_{k+2}, X_{k+1} - T_{k-1};q)
\\
\times 
(-1)^{t_{d-1}} \widehat{p}_{t_{d-1}}(X_{d-1}-T_{d-2}, \alpha_1, 2 T_{d-2} + \widetilde{A}_{d} + d -2;q).
\end{multline}
Upon taking $\widetilde{\bm{x}}_{L} = (L-X_{d-1}, x_1,\ldots, x_{d-1})$ and $\widetilde{\bm{t}}_{L} = (t_1,\ldots, t_{d-1}, L - T_{d-1})$, one has
\begin{align*}
\lim_{L\rightarrow \infty} \Xi_{\widetilde{\bm{t}}_{L}}^{(\bm{\alpha})} (\widetilde{\bm{x}}_{L}) = 
 \mathcal{Q}_{\bm{t}}^{(\bm{\alpha})}(\bm{x}).
\end{align*}
The functions $\mathcal{Q}_{\bm{t}}^{(\bm{\alpha})}(\bm{x})$ satisfy the orthogonality relation
\begin{align*}
\sum_{\bm{x}} \mathcal{Q}_{\bm{t}}^{(\bm{\alpha})}(\bm{x}) \, \mathcal{Q}_{\bm{t}'}^{(\bm{\alpha})}(\bm{x}) = \delta_{\bm{t} \bm{t}'}.
\end{align*}
\end{Proposition}
\begin{proof}
The orthogonality relation follows from \eqref{Ortho-h} and \eqref{Ortho-P}. The calculation of the limit is direct, and involves using \eqref{Useful-Lim} once.
\end{proof}
\subsubsection{A second basis of mixed $q$-Jacobi and $q$-Hahn polynomials}
The functions $\mathcal{Q}_{\bm{t}}^{(\bm{\alpha})}(\bm{x})$ also give rise to a family of multivariate orthogonal polynomials. These polynomials, given below, are a mixture of $q$-Hahn and $q$-Jacobi polynomials. Upon taking $\bm{x} = (x_1,\ldots, x_{d-1})$, $\bm{t} = (t_1,\ldots, t_{d-1})$, and $\bm{\alpha} = (\alpha_1,\ldots, \alpha_{d})$, one can write
\begin{align}
\label{Q-Pols}
\mathcal{Q}_{\bm{t}}^{(\bm{\alpha})}(\bm{x}) = 
\sqrt{\frac{\nu^{(\bm{\alpha})}(\bm{x})}{\tau_{\bm{t}}^{(\bm{\alpha})}}} Q_{\bm{t}}^{(\bm{\alpha})}(\bm{x}),
\end{align}
with $\nu^{(\bm{\alpha})}(\bm{x})$ given by \eqref{Nu-Weight}. The normalization constant $\tau_{\bm{t}}^{(\bm{\alpha})}$ is of the form
\begin{multline*}
\tau_{\bm{t}}^{(\bm{\alpha})} = 
q^{2\binom{T_{d-2}}{2}}
\\
\times \prod_{k=1}^{d-2} 
\frac{1-q^{\widetilde{A}_{k+2} + k }}{1-q^{2 T_k + \widetilde{A}_{k+2} + k}} 
\frac{(q;q)_{t_k} (q^{\widetilde{A}_{k+1} + k};q)_{t_k + 2 T_{k-1}} (q^{\alpha_{k+2}+ 1};q)_{t_k}}{(q^{\widetilde{A}_{k+2}+k};q)_{t_k + 2 T_{k-1}}} q^{t_k (2 T_{k-1} + \widetilde{A}_{k+1}+k)}
\\
\times \frac{1-q^{A_{d}+d-1}}{1-q^{2T_{d-1}+A_d+d-1}} 
\frac{(q;q)_{t_{d-1}} (q^{\widetilde{A}_d + d -1};q)_{t_{d-1} + 2T_{d-2}} (q^{\alpha_1+1};q)_{t_{d-1}} }{(q^{A_d + d -1};q)_{t_{d-1}+2 T_{d-2}}}
q^{T_{d-1}(\alpha_1 + 1)},
\end{multline*}
and the $(d-1)$-variate polynomials $Q_{\bm{t}}^{(\bm{\alpha})}(\bm{x})$ have the expression
\begin{multline*}
Q_{\bm{t}}^{(\bm{\alpha})}(\bm{x}) = 
\prod_{k=1}^{d-2} h_{t_k}(X_{k} - T_{k-1}, 2 T_{k-1} + \widetilde{A}_{k+1} + k -1, \alpha_{k+2}, X_{k+1} - T_{k-1};q)
\\
\times 
(-1)^{t_{d-1}} q^{X_{d-1}T_{d-2}} \,p_{t_{d-1}}(X_{d-1}-T_{d-2}, \alpha_1, 2 T_{d-2} + \widetilde{A}_{d} + d - 2; q).
\end{multline*}
These polynomials are orthogonal with respect to the same measure as the multivariate Jacobi polynomials \eqref{MV-Jacobi}; that is
\begin{align}
\label{Ortho-Q}
\sum_{\bm{x}} \nu^{(\bm{\alpha})}(\bm{x}) \; Q_{\bm{t}}^{(\bm{\alpha})}(\bm{x})
\; Q_{\bm{t}'}^{(\bm{\alpha})}(\bm{x}) = \tau_{\bm{t}}^{(\bm{\alpha})} \delta_{\bm{t} \bm{t}'}.
\end{align}
\begin{Remark}
The functions $Q_{\bm{t}}^{(\bm{\alpha})}(\bm{x})$ are polynomials of total degree $T_{d-1}$ in the variables $q^{\sum_{j=1}^{d-1}x_j}, q^{\sum_{j=2}^{d-1}x_j}, \ldots, q^{x_{d-1}}$ with parameters $\alpha_1, \alpha_2, \ldots, \alpha_{d}$.
\end{Remark}
\section{Interbasis expansion coefficients and $q$-Racah polynomials}\label{se3}
In this section, we shall consider the expansion coefficients between the bases $\Psi_{\bm{n}}^{(\bm{\alpha})}(\bm{y})$ and $\Xi_{\bm{m}}^{(\bm{\alpha})}(\bm{y})$ and show that they are expressed in terms of Gasper \& Rahman's multivariate $q$-Racah polynomials. Given the interpretation of the functions $\Psi_{\bm{n}}^{(\bm{\alpha})}(\bm{y})$ and $\Xi_{\bm{m}}^{(\bm{\alpha})}(\bm{y})$ as nested Clebsch-Gordan coefficients of the $d$-fold tensor product representation $W^{(\bm{\alpha})}$, this will provide an interpretation of the multivariate $q$-Racah polynomials in terms of $3nj$ symbols for $su_q(1,1)$.
\subsection{The main object}\label{ss3.1}
Let $\bm{n} = (n_1,\ldots, n_{d})$, $\bm{m} = (m_1,\ldots, m_{d})$ and $\bm{\alpha} = (\alpha_1,\ldots, \alpha_d)$. We define the functions $\mathcal{R}_{\bm{m}}^{(\bm{\alpha})}(\bm{n})$ as the coefficients that appear in the expansion
\begin{align}
\label{Exp-1-A}
\Xi_{\bm{m}}^{(\bm{\alpha})}(\bm{y}) = \sum_{\bm{n}} \mathcal{R}_{\bm{m}}^{(\bm{\alpha})}(\bm{n}) \;\Psi_{\bm{n}}^{(\bm{\alpha})}(\bm{y}),
\end{align}
where $\Psi_{\bm{n}}^{(\bm{\alpha})}(\bm{y})$ and $\Xi_{\bm{m}}^{(\bm{\alpha})}(\bm{y})$ are respectively given by \eqref{Psi} and \eqref{Xi-Def}. Since the functions $\Psi_{\bm{n}}^{(\bm{\alpha})}(\bm{y})$ and $\Xi_{\bm{m}}^{(\bm{\alpha})}(\bm{y})$ are both orthonormal, one can write the coefficients $\mathcal{R}_{\bm{m}}^{(\bm{\alpha})}(\bm{n})$ as
\begin{align}
\label{R-Def}
\mathcal{R}_{\bm{m}}^{(\bm{\alpha})}(\bm{n}) = \sum_{\bm{y}} \Xi_{\bm{m}}^{(\bm{\alpha})}(\bm{y}) \Psi_{\bm{n}}^{(\bm{\alpha})}(\bm{y}).
\end{align}
Because both $\Psi_{\bm{n}}^{(\bm{\alpha})}(\bm{y})$ and $\Xi_{\bm{m}}^{(\bm{\alpha})}(\bm{y})$ arise from the expansion of joint eigenvectors of $q^{A_0^{[1;d]}}$ and the total Casimir operator $\Gamma^{[1;d]}$, it is clear that expansion coefficients $\mathcal{R}_{\bm{m}}^{(\bm{\alpha})}(\bm{n})$ vanish unless $M_d = N_d$ and $M_{d-1} = N_{d-1}$. These two conditions imply in particular that the coefficients $\mathcal{R}_{\bm{m}}^{(\bm{\alpha})}(\bm{n})$ vanish unless $n_{d}= m_{d}$. Moreover, since the total Casimir operator $\Gamma^{[1;d]}$ commutes with the raising/lowering operators $A_{\pm}^{[1;d]}$, the coefficients $\mathcal{R}_{\bm{m}}^{(\bm{\alpha})}(\bm{n})$ are in fact independent of $n_d$ and $m_d$. Consequently, with $N_{d-1} = M_{d-1}$, one can take $n_{d} = L - N_{d-1}$ and $m_d = L - M_{d-1}$ in \eqref{R-Def} and take $L \rightarrow \infty$ without affecting the value of $\mathcal{R}_{\bm{m}}^{(\bm{\alpha})}(\bm{n})$. This leads to the expression
\begin{align}
\label{R-Def-2}
\mathcal{R}_{\bm{m}}^{(\bm{\alpha})}(\bm{n}) = \sum_{\bm{x}}  \mathcal{Q}_{\bm{m}}^{(\bm{\alpha})}(\bm{x})
\mathcal{J}_{\bm{n}}^{(\bm{\alpha})}(\bm{x}),
\end{align}
where the summation runs over all multi-indices $\bm{x}=(x_1,\ldots, x_{d-1})$ of non-negative integers and where we have taken $\bm{m} = (m_1,\ldots,m_{d-1})$ and $\bm{n} = (n_1,\ldots, n_{d-1})$, allowed given the independence of $\mathcal{R}_{\bm{m}}^{(\bm{\alpha})}(\bm{n})$ on the last quantum numbers $m_{d}$ and $n_d$. The functions $\mathcal{J}_{\bm{n}}^{(\bm{\alpha})}(\bm{x})$ and $\mathcal{Q}_{\bm{m}}^{(\bm{\alpha})}(\bm{x})$ are given by \eqref{J-Def} and \eqref{Q-Def}, respectively. One has also
\begin{align}
\label{Exp-2-A}
\mathcal{Q}_{\bm{m}}^{(\bm{\alpha})}(\bm{x}) = \sum_{\bm{n}} \mathcal{R}_{\bm{m}}^{(\bm{\alpha})}(\bm{n})\,\mathcal{J}_{\bm{n}}^{(\bm{\alpha})}(\bm{x}),
\end{align}
where the sum runs over the multi-indices $\bm{n} = (n_1,\ldots,n_{d-1})$ such that $N_{d-1} = M_{d-1}$. Here also, the coefficients are independent of the coordinates $\bm{x}$.

\subsection{Connection coefficients between multivariate polynomials}
The functions $\mathcal{R}_{\bm{m}}^{(\bm{\alpha})}(\bm{n})$ are connection coefficients between families of multivariate orthogonal polynomials. Indeed, it follows from \eqref{Psi-Pols}, \eqref{Xi-Pols} and \eqref{Exp-1-A} that
\begin{align}\label{eq-Racah}
\sqrt{\frac{1}{\Omega_{\bm{m}}^{(\bm{\alpha})}}}\;
G_{\bm{m}}^{(\bm{\alpha})}(\bm{y}) = \sum_{\bm{n}} \mathcal{R}_{\bm{m}}^{(\bm{\alpha})}(\bm{n})\; \sqrt{\frac{1}{\Lambda_{\bm{n}}^{(\bm{\alpha})}}}\; H_{\bm{n}}^{(\bm{\alpha})}(\bm{y}),
\end{align}
where the sum is over the multi-indices $\bm{n}$ such that $N_{d} = M_{d}$ and $N_{d-1} = M_{d-1}$. One can also use the orthogonality relation for $H_{\bm{n}}^{(\bm{\alpha})}(\bm{y})$ to get the formula
\begin{align*}
\mathcal{R}_{\bm{m}}^{(\bm{\alpha})}(\bm{n}) = \sqrt{\frac{1}{\Lambda_{\bm{n}}^{(\bm{\alpha})}\Omega_{\bm{m}}^{(\bm{\alpha})}}} \sum_{\bm{y}} \rho^{(\bm{\alpha})}(\bm{y}) H_{\bm{n}}^{(\bm{\alpha})}(\bm{y})\; G_{\bm{m}}^{(\bm{\alpha})}(\bm{y}),
\end{align*}
where the sum is restricted to all $\bm{y}$ such that $Y_d = N_d = M_d$. Furthermore, from \eqref{J-Pols}, \eqref{Q-Pols} and \eqref{Exp-2-A} one finds that
\begin{align*}
\sqrt{\frac{1}{\tau_{\bm{m}}^{(\bm{\alpha})}}}\;
Q_{\bm{m}}^{(\bm{\alpha})}(\bm{x}) = \sum_{\bm{n}} \mathcal{R}_{\bm{m}}^{(\bm{\alpha})}(\bm{n})\; \sqrt{\frac{1}{\iota_{\bm{n}}^{(\bm{\alpha})}}}\; J_{\bm{n}}^{(\bm{\alpha})}(\bm{x}),
\end{align*}
which is equivalent to
\begin{align*}
\mathcal{R}_{\bm{m}}^{(\bm{\alpha})}(\bm{n}) = \sqrt{\frac{1}{\tau_{\bm{m}}^{(\bm{\alpha})}\iota_{\bm{n}}^{(\bm{\alpha})}}}
\sum_{\bm{x}} \nu^{(\bm{\alpha})}(\bm{x}) \; Q_{\bm{m}}^{(\bm{\alpha})}(\bm{x})\, J_{\bm{n}}^{(\bm{\alpha})}(\bm{x}).
\end{align*}
As can be seen, the coefficients $\mathcal{R}_{\bm{m}}^{(\bm{\alpha})}(\bm{n})$ serve as connection coefficients between bases of multivariate $q$-Hahn or $q$-Jacobi polynomials. Since these bases are themselves orthogonal, it follows from elementary linear algebra that the functions $\mathcal{R}_{\bm{m}}^{(\bm{\alpha})}(\bm{n})$ satisfy the orthogonality relations
\begin{subequations}
\begin{align}
\label{Ortho-a}
&\sum_{\bm{n}}
\mathcal{R}_{\bm{m}}^{(\bm{\alpha})}(\bm{n})
\mathcal{R}_{\bm{m}'}^{(\bm{\alpha})}(\bm{n})=\delta_{\bm{m}\bm{m}'},
\\
&
\label{Ortho-b}
\sum_{\bm{m}}\mathcal{R}_{\bm{m}}^{(\bm{\alpha})}(\bm{n})
\mathcal{R}_{\bm{m}}^{(\bm{\alpha})}(\bm{n}')=\delta_{\bm{n}\bm{n}'}.
\end{align}
\end{subequations}
These relations are meaningful when the indices $\bm{n}$, $\bm{n}'$, $\bm{m}$ and $\bm{m}'$ are such that $N_{d-1} = N_{d-1}' = M_{d-1} = M_{d-1}'$, which insure that $\mathcal{R}_{\bm{m}}^{(\bm{\alpha})}(\bm{n})$ does not trivially vanish.

\subsection{The $d=3$ case: one-variable $q$-Racah polynomials}
Let us now consider the expansion coefficients $\mathcal{R}_{\bm{m}}^{(\bm{\alpha})}(\bm{n})$ for $d=3$. As we noted above, the coefficients $\mathcal{R}_{\bm{m}}^{(\bm{\alpha})}(\bm{n})$ vanish unless $M_2=N_2$ and $m_3=n_3$, and are independent of the quantum numbers $m_3$ and $n_3$. Throughout this subsection, we assume that $\bm{m}$ and $\bm{n}$ satisfy these conditions, and to simplify the notation, we will omit $m_3$ and $n_3$ when we display the coefficients, i.e. we will write simply $\mathcal{R}_{m_1,m_2}^{(\alpha_1,\alpha_2,\alpha_3)}(n_1,n_2)$ instead of $\mathcal{R}_{m_1,m_2,m_3}^{(\alpha_1,\alpha_2,\alpha_3)}(n_1,n_2,n_3)$.

Upon writing \eqref{Exp-2-A} explicitly, one gets
\begin{multline}
\label{Expansion-1}
(-1)^{m_2}q^{m_1(x_1+x_2)}
h_{m_1}(x_1, \alpha_2,\alpha_3,x_1+x_2;q) 
p_{m_2}(x_1+x_2-m_1,\alpha_1, 2m_1+\alpha_2+\alpha_3+1;q)
\\
= \sum_{\bm{n}} 
\mathcal{R}_{m_1,m_2}^{(\alpha_1,\alpha_2,\alpha_3)}(n_1,n_2)
\sqrt{\frac{\tau_{m_1,m_2}^{(\alpha_1,\alpha_2,\alpha_3)}}{\iota_{n_1,n_2}^{(\alpha_1,\alpha_2,\alpha_3)}}}
\\
\times
(-1)^{n_1 + n_2} q^{n_1 x_2} p_{n_1}(x_1,\alpha_1,\alpha_2;q)  p_{n_2}(x_2,2n_1 + \alpha_1 + \alpha_2 + 1, \alpha_3;q),
\end{multline}
where $h_{n}(x,\alpha,\beta,N;q)$ and $p_{n}(x,\alpha,\beta;q)$ are the $q$-Hahn and $q$-Jacobi polynomials defined in \eqref{q-Hahn} and \eqref{q-Jacobi}. We now set $x_1 = (u-v)/2$,  $x_2 = (u+v)/2$ as well as $t = q^{u+1}$, and we consider the limit of \eqref{Expansion-1} as $v\rightarrow \infty$. Recalling that $0 < q < 1$, one can use the transformation formula \cite{Gasper2004}
\begin{align*}
{}_3\phi_2\left(\genfrac{}{}{0pt}{}{q^{-n}, c/b, 0}{c, cq/bz};q,q\right)
=
\frac{(bz/c;q)_{\infty}}{(bzq^{-n}/c;q)_{\infty}}\;
{}_2\phi_1\left(\genfrac{}{}{0pt}{}{q^{-n}, b}{c};q,z\right),
\end{align*}
to find that the left-hand side of \eqref{Expansion-1} becomes under the limit $v\rightarrow \infty$
\begin{multline*}
\lim_{v\rightarrow \infty} (-1)^{m_2}q^{m_1 u}
h_{m_1}((u-v)/2, \alpha_2,\alpha_3,u;q) 
p_{m_2}(u-m_1,\alpha_1, 2m_1+\alpha_2+\alpha_3+1;q)
\\
= (-1)^{m_1 + m_2} \; q^{\binom{m_1}{2}} (q^{\alpha_2 +1};q)_{m_1} (q^{\alpha_1 +1};q)_{m_2}
\\
\times
{}_2\phi_1\left(\genfrac{}{}{0pt}{}{q^{-m_1}, q^{-m_1-\alpha_3}}{q^{\alpha_2+1}};q,q^{m_1+\alpha_2+\alpha_3 +1}t\right)\;
{}_2\phi_1\left(\genfrac{}{}{0pt}{}{q^{-m_2}, q^{m_2+2m_1+\alpha_1+\alpha_2 + \alpha_3 +2}}{q^{\alpha_1+1}};q,q^{-m_1}t\right).
\end{multline*}
Upon taking the same limit on the right-hand side of \eqref{Expansion-1}, one easily finds
\begin{multline*}
\lim_{v\rightarrow \infty} (-1)^{n_1 + n_2} q^{n_1 (u+v)/2} p_{n_1}((u-v)/2,\alpha_1,\alpha_2;q)  p_{n_2}((u+v)/2,2n_1 + \alpha_1 + \alpha_2 + 1, \alpha_3;q)
\\
= (-1)^{n_1 + n_2}\frac{(q^{-n_1};q)_{n_1} (q^{n_1+\alpha_1+\alpha_2+1};q)_{n_1} (q^{2n_1 +\alpha_1 +\alpha_2 +1};q)_{n_2}}{(q;q)_{n_1}}t^{n_1}.
\end{multline*}
As a consequence, one has
\begin{multline}
\label{Expansion-2}
{}_2\phi_1\left(\genfrac{}{}{0pt}{}{q^{-m_1}, q^{-m_1-\alpha_3}}{q^{\alpha_2+1}};q,q^{m_1+\alpha_2+\alpha_3 +1}t\right)\;
{}_2\phi_1\left(\genfrac{}{}{0pt}{}{q^{-m_2}, q^{m_2 + 2m_1+\alpha_1+\alpha_2 + \alpha_3  +2}}{q^{\alpha_1+1}};q,q^{-m_1}t\right)
\\
= \sum_{n_1, n_2} \mathcal{R}_{m_1,m_2}^{(\alpha_1,\alpha_2,\alpha_3)}(n_1, n_2)
\\
\times
\left[ \sqrt{\frac{\tau_{m_1,m_2}^{(\alpha_1,\alpha_2,\alpha_3)}}{\iota_{n_1,n_2}^{(\alpha_1,\alpha_2,\alpha_3)}}}
\frac{(q^{-n_1};q)_{n_1} (q^{n_1+\alpha_1+\alpha_2+1};q)_{n_1} (q^{2n_1 +\alpha_1 +\alpha_2 +1};q)_{n_2}}{(q;q)_{n_1} q^{\binom{m_1}{2}} (q^{\alpha_1 +1};q)_{m_1} (q^{\alpha_2 +1};q)_{m_2}}
\right] t^{n_1},
\end{multline}
where the sum runs over all $n_1$, $n_2$ such that $n_1 + n_2 = m_1 + m_2$. The generating relation \eqref{Expansion-2} can be seen to coincide with that of the one-variable $q$-Racah polynomials. Indeed, let $r_{n}(x,\mathfrak{a}, \mathfrak{b}, \mathfrak{c},N;q)$ be the $q$-Racah polynomials \cite{Koekoek2010}
\begin{multline*}
r_{n}(x,\mathfrak{a}, \mathfrak{b}, \mathfrak{c},N;q)=
\\
(q^{\mathfrak{a+1}};q)_{n} (q^{\mathfrak{b} + \mathfrak{c} +1};q)_{n}(q^{-N};q)_{n} q^{n(N-\mathfrak{c})/2}
\;
{}_4\phi_3\left(\genfrac{}{}{0pt}{}{q^{-n}, q^{n+\mathfrak{a}+\mathfrak{b}+1}, q^{-x}, q^{x+\mathfrak{c}-N}}{q^{\mathfrak{a}+1}, q^{\mathfrak{b}+\mathfrak{c}+1}, q^{-N}};q,q \right).
\end{multline*}
Let $\sigma(x,\mathfrak{a}, \mathfrak{b}, \mathfrak{c},N;q)$ be defined as 
\begin{align*}
\sigma(x,\mathfrak{a}, \mathfrak{b}, \mathfrak{c},N;q) =
\frac{1-q^{2x+\mathfrak{c}-N}}{1-q^{\mathfrak{c}-N}}
\frac{(q^{\mathfrak{c}-N};q)_{x}(q^{\mathfrak{a}+1};q)_{x} (q^{\mathfrak{b}+\mathfrak{c}+ 1};q)_{x}(q^{-N};q)_{x}}{(q;q)_x (q^{\mathfrak{c}-\mathfrak{a}-N};q)_{x}(q^{-\mathfrak{b}-N};q)_{x}(q^{\mathfrak{c}+1};q)_{x}} q^{-x(\mathfrak{a} + \mathfrak{b}+1)},
\end{align*}
and let $\upsilon(n,\mathfrak{a}, \mathfrak{b}, \mathfrak{c}, N;q)$ have the expression
\begin{multline*}
\upsilon(n,\mathfrak{a}, \mathfrak{b}, \mathfrak{c}, N;q) =
\frac{(q^{-\mathfrak{c}};q)_{N} (q^{\mathfrak{a} + \mathfrak{b}+2};q)_{N}}{(q^{\mathfrak{a}-\mathfrak{c} + 1};q)_{N} (q^{\mathfrak{b}+1};q)_{N}}
(q;q)_n(q^{\mathfrak{a}+1};q)_n (q^{\mathfrak{b}+1};q)_n
\\
\times  (q^{\mathfrak{a}-\mathfrak{c}+1};q)_n
(q^{\mathfrak{b}+\mathfrak{c}+1};q)_n (q^{-N};q)_n \frac{1-q^{\mathfrak{a}+\mathfrak{b}+1}}{1-q^{2n +\mathfrak{a} + \mathfrak{b} + 1}} \frac{(q^{N + \mathfrak{a} + \mathfrak{b} + 2};q)_n}{(q^{\mathfrak{a} + \mathfrak{b} +1};q)_{n}}.
\end{multline*}
The $q$-Racah functions $\widehat{r}_n(x,\mathfrak{a}, \mathfrak{b}, \mathfrak{c},N;q)$ defined as
\begin{align*}
\widehat{r}_n(x,\mathfrak{a}, \mathfrak{b}, \mathfrak{c},N;q) = \sqrt{\frac{\sigma(x,\mathfrak{a}, \mathfrak{b}, \mathfrak{c},N;q)}{\upsilon(n,\mathfrak{a}, \mathfrak{b}, \mathfrak{c}, N;q)}}
\; r_n(x,\mathfrak{a}, \mathfrak{b}, \mathfrak{c},N;q),
\end{align*}
satisfy the orthogonality relation
\begin{align*}
\sum_{x=0}^{N}\widehat{r}_n(x,\mathfrak{a}, \mathfrak{b}, \mathfrak{c},N;q)\;\widehat{r}_{n'}(x,\mathfrak{a}, \mathfrak{b}, \mathfrak{c},N;q) = \delta_{nn'}.
\end{align*}
Their generating relation reads
\begin{multline}
\label{Gen-Racah}
{}_2\phi_1\left(\genfrac{}{}{0pt}{}{q^{-n}, q^{-n-\mathfrak{b}}}{q^{\mathfrak{a}+1}};q,q^{n+\mathfrak{a}+\mathfrak{b}+1}t\right)\;
{}_2\phi_1\left(\genfrac{}{}{0pt}{}{q^{n-N}, q^{n+\mathfrak{b} + \mathfrak{c} +1}}{q^{\mathfrak{c}-\mathfrak{a}-N}};q,q^{-n}t\right)
\\
=\sum_{x=0}^{N} \frac{(q^{\mathfrak{b}+\mathfrak{c}+1};q)_x(q^{-N};q)_x}{(q;q)_x (q^{\mathfrak{c}-\mathfrak{a}-N};q)_x} \frac{q^{n(\mathfrak{c}-N)/2}}{(q^{\mathfrak{a+1}};q)_{n} (q^{\mathfrak{b} + \mathfrak{c} +1};q)_{n}(q^{-N};q)_{n} } r_{n}(x,\mathfrak{a},\mathfrak{b}, \mathfrak{c},N;q)\;t^x.
\end{multline}
Upon comparing \eqref{Expansion-2} with \eqref{Gen-Racah}, we obtain the following.
\begin{Proposition}\label{Prop5}
When $d=3$, the expansion coefficients $\mathcal{R}_{m_1, m_2}^{(\alpha_1,\alpha_2,\alpha_3)}(n_1, n_2)$ can be expressed in terms of the univariate $q$-Racah polynomials. Explicitly, one has
\begin{align*}
\mathcal{R}_{m_1, m_2}^{(\alpha_1,\alpha_2,\alpha_3)}(n_1, n_2) = \delta_{N_2 M_2}\;
(-1)^{m_1}\widehat{r}_{m_1}(n_1, \alpha_2, \alpha_3, n_1 + n_2 + \alpha_1 + \alpha_2 + 1, n_1 + n_2;q).
\end{align*}
\end{Proposition}
\begin{proof}
The result follows from the above discussion and from comparing \eqref{Expansion-2} and \eqref{Gen-Racah}. It is seen that the two coincide, up to normalization factors, if one takes
\begin{align*}
n = m_1, \quad x = n_1, \quad N = n_1 + n_2,\quad \mathfrak{a} = \alpha_2,\quad \mathfrak{b} = \alpha_3,\quad \mathfrak{c} = N + \alpha_1 +\alpha_2 +1.
\end{align*}
The calculation of the normalization factors is straightforward.
\end{proof}
Let us note that our derivation of Proposition \ref{Prop5} is much simpler than the one presented in \cite{Dunkl1980}, which used different methods. Moreover, our construction has the advantage of unifying two of the main interpretations of the one-variable $q$-Racah polynomials: 1) their interpretation as connection coefficients for 2-variable $q$-Hahn or $q$-Jacobi polynomials, and 2) their interpretation as $6j$ coefficients for positive-discrete series representations of $su_q(1,1)$. 
In the next section, we shall give a new interpretation in connection with $q$-deformed quantum superintegrable systems.

\begin{Remark}\label{generate-Racah}
We can use special values of the $\bm{y}$ variables in equation \eqref{eq-Racah} to obtain other generating functions for the $q$-Racah polynomials. First, note that by using the $q$-analog of the  Chu-Vandermonde identity, the $q$-Hahn polynomials in \eqref{q-Hahn} reduce to simple products of $q$-Pochhammer symbols when $x=-(\alpha+1)$ or $x=N$:
\begin{align*}
&h_{n}(-\alpha-1,\alpha,\beta,N;q) =(q^{\alpha+1};q)_n(q^{-N-\alpha-\beta-n-1};q)_n\,q^{n(n+\alpha+\beta+1)},\\
&h_{n}(N,\alpha,\beta,N;q) =(q^{-N};q)_n(q^{-n-\beta};q)_n\,q^{n(n+\alpha+\beta+1)}.
\end{align*}
If we fix $N_3=Y_3=L$, with $y_1=-\alpha_1-1$, $y_2=L+\alpha_1+1$, $y_3=0$ and replace $L$ with $w$, where $w=q^{-L-\alpha_1-1}$, then equations \eqref{MV-q-Hahn}, \eqref{G-Pols},  together with the above formulas show that
\begin{align*}
&H_{\bm{n}}^{(\bm{\alpha})}(-\alpha_1-1,L+\alpha_1+1,0)\equiv  (wq^{-\alpha_2-n_1};q)_{n_1} (wq^{\alpha_1+n_1+1};q)_{n_2},\\
&G_{\bm{m}}^{(\bm{\alpha})}(-\alpha_1-1,L+\alpha_1+1,0) \equiv  (w;q)_{m_1} (wq^{-(\alpha_2+\alpha_3+m_1+m_2+1)};q)_{m_2},
\end{align*}
where $\equiv$ means that the equality holds up to a multiple independent of $w$. For generic values of $\alpha_1$, $\alpha_2$, $\alpha_3$, the polynomials of $w$ in each of the sets: 
\begin{itemize}
\item $\left\{(w;q)_{m_1} (wq^{-(\alpha_2+\alpha_3+m_1+m_2+1)};q)_{m_2}:(m_1,m_2)\in\mathbb{N}_0^2\right\}$, and 
\item $\left\{ (wq^{-\alpha_2-n_1};q)_{n_1} (wq^{\alpha_1+n_1+1};q)_{n_2}:(n_1,n_2)\in\mathbb{N}_0^2\right\}$
\end{itemize} 
are linearly independent. Therefore, if we substitute $y_1=-\alpha_1-1$, $y_2=L+\alpha_1+1$, $y_3=0$ into  \eqref{eq-Racah} and use the above formulas, we obtain the $q$-Racah polynomials  $\mathcal{R}_{m_1, m_2}^{(\alpha_1,\alpha_2,\alpha_3)}(n_1, n_2)$ as connecting coefficients between the different bases $\{(wq^{-\alpha_2-n_1};q)_{n_1} (wq^{\alpha_1+n_1+1};q)_{n_2}\}$ and $\{(w;q)_{m_1} (wq^{-(\alpha_2+\alpha_3+m_1+m_2+1)};q)_{m_2}\}$ of polynomials in  $w$. This connection was observed in \cite{Koelink1998} in a different setting and notations, see Remark 4.11(iii) on page 815.
\end{Remark}

\subsection{Multivariate $q$-Racah polynomials and $3nj$ coefficients for $su_q(1,1)$}
We shall now generalize the result of Proposition \ref{Prop5} to the multivariate case. While this result can be proven by induction, the proof is cumbersome and fails to provide additional insight into the structure of the coefficients $\mathcal{R}_{\bm{m}}^{(\bm{\alpha})}(\bm{n})$. In the following, we thus opt to construct the expression for the connection coefficients $\mathcal{R}_{\bm{m}}^{(\bm{\alpha})}(\bm{n})$ in the $d=4$ case in terms of 2-variable $q$-Racah polynomials. The result is then seen to extend directly to an arbitrary number of variables.

\subsubsection{The $d=4$ case}
Consider the basis functions $\Xi_{\bm{m}}^{(\bm{\alpha})}(\bm{y})$ when $d=4$. One has
\begin{multline}
\label{Step-0}
\Xi_{\bm{m}}^{(\bm{\alpha})}(\bm{y}) = 
\widehat{h}_{m_1}(y_2, \alpha_2, \alpha_3, y_{23};q)
\\
\times \widehat{h}_{m_2}(y_{23}-m_1, 2m_1 +\alpha_{23} +1, \alpha_4, y_{234} -m_{1};q)
\\
\times \widehat{h}_{m_3}(y_1, \alpha_1, 2m_{12} + \alpha_{234} + 2, y_{1234} -m_{12};q),
\end{multline}
where we used the notation $x_{ij} = x_i + x_j$ and $x_{ijk}=x_i  + x_j + x_k$. If one defines $\widetilde{y}_2 = y_{23} -m_1$, $\widetilde{\alpha}_2 = 2m_1 + \alpha_{23} +1$, $\widetilde{\alpha}_3 = \alpha_4$, $\widetilde{y}_1 = y_1$ and $\widetilde{y}_3 = y_4$, it is seen that \eqref{Step-0} reads
\begin{multline}
\label{Step-1}
\Xi_{\bm{m}}^{(\bm{\alpha})}(\bm{y}) = 
\widehat{h}_{m_1}(y_2, \alpha_2, \alpha_3, y_{23};q)
\\
\times \widehat{h}_{m_2}(\widetilde{y}_2, \widetilde{\alpha}_2, \widetilde{\alpha_3}, \widetilde{y}_{23};q)\, \widehat{h}_{m_3}(\widetilde{y}_1,\alpha_1,2m_2 +\widetilde{\alpha}_{23} + 1, \widetilde{y}_{123} - m_2).
\end{multline}
It is observed that the last two $q$-Hahn functions in \eqref{Step-1} have the appropriate form to apply Proposition \ref{Prop5}. We thus have
\begin{multline}
\label{Step-2}
\Xi_{\bm{m}}^{(\bm{\alpha})}(\bm{y}) =
\\
\sum_{\substack{n_2, n_3\\ n_2 + n_3 =m_2 + m_3}}
(-1)^{m_2} \widehat{r}_{m_2}(n_2,2m_1 + \alpha_{23} +1, \alpha_4, n_{23} +2m_1 + \alpha_{123} +2, n_{23};q)
\\
\times \widehat{h}_{m_1}(y_2, \alpha_2, \alpha_3, y_{23};q)\;
\widehat{h}_{n_2}(y_1,\alpha_1, 2m_1+\alpha_{23}+1,y_{123}-m_1;q)
\\
\times \widehat{h}_{n_3}(y_{123} -m_1 -n_2, 2n_2 + 2m_1 +\alpha_{123}+2, \alpha_4, y_{1234} -m_1 -n_2;q).
\end{multline}
It is further seen that the first two $q$-Hahn functions now have the appropriate form to apply Proposition \ref{Prop5} a second time. This leads to the following:
\begin{multline}
\label{Step-4}
\Xi_{\bm{m}}^{(\bm{\alpha})}(\bm{y}) = \sum_{\substack{n_2, n_3\\ n_2 + n_3 =m_2 + m_3}} \sum_{\substack{n_1', n_2'\\ n_1' + n_2' = m_1 + n_2}} 
\\
\times (-1)^{m_2} \widehat{r}_{m_2}(n_2,2m_1 + \alpha_{23} +1, \alpha_4, n_{23} +2m_1 + \alpha_{123} +2, n_{23};q)
\\
\times (-1)^{m_1} \widehat{r}_{m_1}(n_1',\alpha_2, \alpha_3, m_1 + n_2 +\alpha_{12}+1, m_1 + n_2;q)
\\
\times \widehat{h}_{n_1'}(y_1, \alpha_1, \alpha_2, y_1 + y_2;q) \widehat{h}_{n_2'}(y_{12}-n_1',2n_1'+\alpha_{12}+1,\alpha_3,y_{123}-n_1';q)
\\
\times \widehat{h}_{n_3}(y_{123}-m_1 -n_2, 2n_2 + 2m_1 +\alpha_{123}+2,\alpha_4,y_{1234}-m_1 -n_2;q).
\end{multline}
Upon taking $n_2 = n_1'+n_2'-m_1$ and then renaming $n_1' \rightarrow n_1$, $n_2'\rightarrow n_2$, one finds
\begin{align*}
\Xi_{\bm{m}}^{(\bm{\alpha})}(\bm{y}) = \sum_{\bm{n}} \mathcal{R}_{\bm{m}}^{(\bm{\alpha})}(\bm{n})\; \Psi_{\bm{n}}^{(\bm{\alpha})}(\bm{y}),
\end{align*}
where the summation runs over the multi-indices  $\bm{n}$ such that $n_1 + n_2 + n_3 = m_1 + m_2 + m_3$ and where 
\begin{multline}
\label{Final-1}
\mathcal{R}_{\bm{m}}^{(\bm{\alpha})}(\bm{n}) = (-1)^{m_1 + m_2} 
\widehat{r}_{m_1}(n_1,\alpha_2,\alpha_3,n_{12} + \alpha_{12} +1, n_{12};q)
\\
\times \widehat{r}_{m_2}(n_{12}-m_1, 2m_1 + \alpha_{23}+1, \alpha_4, n_{123}+m_1 +\alpha_{123}+2, n_{123}-m_1;q).
\end{multline}
\subsubsection{The general result}
Let us now state the general result giving the explicit expression of the coefficients $\mathcal{R}_{\bm{m}}^{(\bm{\alpha})}(\bm{n})$ for an arbitrary number of variables.
\begin{Proposition}
Let $\bm{m} = (m_1,\ldots, m_{d})$, $\bm{n} = (n_1,\ldots, n_{d})$ and $\bm{\alpha} = (\alpha_1,\ldots, \alpha_{d})$. The interbasis expansion coefficients $\mathcal{R}_{\bm{m}}^{(\bm{\alpha})}(\bm{n})$ defined in \eqref{Exp-1-A} have the explicit expression
\begin{multline}
\label{3nj-Coef}
\mathcal{R}_{\bm{m}}^{(\bm{\alpha})}(\bm{n}) = \delta_{m_{d} n_{d}} \delta_{M_{d-1} N_{d-1}}\prod_{k=1}^{d-2} (-1)^{m_k} 
\\
\times
\widehat{r}_{m_k}(N_{k} - M_{k-1}, 2 M_{k-1} + \widetilde{A}_{k+1} + k -1, \alpha_{k+2},
N_{k+1} + M_{k-1} + A_{k+1} + k, N_{k+1} - M_{k-1};q).
\end{multline}
\end{Proposition}
\begin{proof}
By induction; following the steps outlined in the previous subsection. The basic case $d=3$ is proven in Proposition \ref{Prop5}. Suppose that the result holds at level $d-1$. Consider the basis functions $\Xi_{\bm{m}}^{(\alpha)}(\bm{y})$ at level $d$. One can write
\begin{align*}
\Xi_{\bm{m}}^{(\alpha)}(\bm{y}) = \widehat{h}_{m_1}(y_2, \alpha_2, \alpha_3,y_2 + y_3;q)
\times 
\left[
\Xi_{m_2,\ldots, m_{d}}^{(\widetilde{\alpha}_1,\ldots, \widetilde{\alpha}_{d-1})} (\widetilde{y}_1,\ldots, \widetilde{y}_{d-1})
\right],
\end{align*}  
with 
\begin{align*}
&\widetilde{\alpha}_1 = \alpha_1,\qquad \widetilde{\alpha}_2 = 2m_1 + \alpha_{23} +1, \qquad 
\widetilde{\alpha}_k = \alpha_{k+1},
\\
&\widetilde{y}_1 = y_1, \qquad \widetilde{y}_2 = y_2 + y_3 -m_1,\qquad \widetilde{y}_{k} = y_{k+1}.
\end{align*}
Then, one uses the induction hypothesis to develop the functions $\Xi_{m_2,\ldots, m_{d}}^{(\widetilde{\alpha}_1,\ldots, \widetilde{\alpha}_{d-1})} (\widetilde{y}_1,\ldots, \widetilde{y}_{d-1})$ in the basis functions $\Psi_{n_2,\ldots, n_{d}}^{(\widetilde{\alpha}_1,\ldots, \widetilde{\alpha}_{d-1})}(\widetilde{y}_1,\ldots, \widetilde{y}_{d-1})$. The procedure is completed by applying Proposition \ref{Prop5} one last time. 
\end{proof}
We have thus obtained the explicit expression \eqref{3nj-Coef} for the expansion coefficients $\mathcal{R}_{\bm{m}}^{(\bm{\alpha})}(\bm{n})$ between the $q$-Hahn bases $\Psi_{\bm{n}}^{(\bm{\alpha})}(\bm{y})$ and $\Xi_{\bm{m}}^{(\bm{\alpha})}(\bm{y})$ defined in \eqref{Psi} and \eqref{Xi-Def}. These are also the expansion coefficients between the $q$-Jacobi bases $\mathcal{J}_{\bm{n}}^{(\bm{\alpha})}(\bm{y})$ and $\mathcal{Q}_{\bm{m}}^{(\bm{\alpha})}(\bm{y})$ defined in \eqref{J-Def} and \eqref{Q-Def}. Moreover, since these coefficients are the overlaps between basis vectors corresponding to irreducible decompositions of the multifold tensor product representation $W^{(\bm{\alpha})}$ of $su_q(1,1)$, the coefficients $\mathcal{R}_{\bm{m}}^{(\bm{\alpha})}(\bm{n})$ can also be considered as particular $3nj$-coefficients.

\begin{Remark}
Using \eqref{3nj-Coef}, we show in the next subsection that the  coefficients $\mathcal{R}_{\bm{m}}^{(\bm{\alpha})}(\bm{n})$ can be expressed in terms of multivariate $q$-Racah polynomials  defined by Gasper and Rahman in \cite{Gasper&Rahman_2007}. Similarly to the $d=3$ case discussed in Remark~\ref{generate-Racah}, we can use special values of the $\bm{y}$ variables in equation \eqref{eq-Racah} to obtain different identities for these polynomials. However, if we fix the values of $y_1,\dots,y_{d}$, so that the $q$-Hahn polynomials reduce to products of $q$-Pochhammer symbols, there will be just one free variable left, which is not sufficient to characterize the $q$-Racah polynomials in the multivariate setting.
\end{Remark}

\subsection{Multivariate $q$-Racah polynomials}
The coefficients $\mathcal{R}_{\bm{m}}^{(\bm{\alpha})}(\bm{n})$ can be expressed in terms of the multivariate $q$-Racah polynomials introduced by Gasper and Rahman in \cite{Gasper&Rahman_2007}. In the $s$-variable case, these polynomials can be written as
\begin{multline}
\label{GR-Racah}
Z_{\bm{\ell}}(\bm{y};\bm{\beta},M;q) = 
\\
\prod_{k=1}^{s}
r_{\ell_k}(y_k - L_{k-1}, 2 L_{k-1} + \beta_{k} - \beta_0 - 1, \beta_{k+1} - \beta_k - 1, y_{k+1} + L_{k-1} + \beta_{k}, y_{k+1} - L_{k-1};q),
\end{multline}
where $y_0 = 0$, $y_{s+1} = M$, and $M\in \mathbb{N}$ and $\bm{\beta} = (\beta_0,\beta_1,\ldots, \beta_{s+1})$ are parameters. It is not hard to see that $Z_{\bm{\ell}}(\bm{y};\bm{\beta},M;q)$ is a polynomial of total degree $L_s$ in the variables $z_k = q^{-y_k} + \beta_k q^{y_k}$. Moreover, they are orthogonal on  the simplex 
\begin{align*}
W_M=\{\bm{y} \in \mathbb{N}_0^s: 0 \le y_1 \le y_2 \le \cdots \le y_s \le  M\},
\end{align*}
with respect to the weight function
\begin{equation}\label{R-weight}
\chi^{(\bm{\beta})}(\bm{y})=\prod_{k=0}^{s}\frac{(q^{\beta_{k+1}-\beta_{k}};q)_{y_{k+1}-y_{k}}
	(q^{\beta_{k+1}};q)_{y_{k+1}+y_{k}}}
{(q;q)_{y_{k+1}-y_{k}}(q^{\beta_{k}+1};q)_{y_{k+1}+y_{k}}}
\prod_{k=1}^{s}\frac{1-q^{2y_k+\beta_{k}}} {1-q^{\beta_k}}q^{y_k(\beta_{k-1}-\beta_{k})}.
\end{equation}
The parametrization in \cite{Gasper&Rahman_2007} can be obtained by taking 
$a_1=q^{\beta_1}$, $a_k=q^{\beta_{k}-\beta_{k-1}}$ for $k=2,\dots,s+1$ and 
$b=q^{\beta_1-\beta_0-1}$.
The square of the norm is
\begin{multline}
\label{R-norm}
\Upsilon_{\bm{\ell}}^{(\bm{\beta})}= \langle Z_{\bm{\ell}}(\bm{y};\bm{\beta},M;q), Z_{\bm{\ell}}(\bm{y};\bm{\beta},M;q)\rangle
\\
=q^{-M(2L_s+\beta_{s})+L_s(L_{s}-1)+\beta_0(M-L_s)}
\frac{(q^{\beta_{s+1}};q)_{L+L_s}(q^{\beta_{s+1}-\beta_{0}};q)_{L+L_s} }{(q;q)_{L-L_s}(q^{\beta_{0}+1};q)_{L-L_s}}
\\
\times\prod_{k=1}^{s}\frac{(q;q)_{\ell_k}(q^{\beta_{k+1}-\beta_{k}};q)_{\ell_k}(q^{\beta_{k}-\beta_{0}};q)_{L_k+L_{k-1}}}
{(q^{\beta_{k+1}-\beta_0-1};q)_{L_k+L_{k-1}}}\frac{1-q^{\beta_{k+1}-\beta_0-1}}{1-q^{\beta_{k+1}-\beta_0-1+2L_{k}}}.
\end{multline}
Upon comparing \eqref{GR-Racah} with \eqref{3nj-Coef}, it is seen that the coefficients \eqref{3nj-Coef} will be proportional to the polynomials \eqref{GR-Racah} if one takes $s= d-2$ and
\begin{align}
\label{Para}
\begin{aligned}
& \ell_i = m_i, \qquad i=1,\ldots, d-2,
\\
& y_i = N_{i}, \qquad i=1,\ldots, d-1,
\\
&\beta_{k} = A_{k+1} + k,\qquad k=0,1,\ldots,d-1.
\end{aligned}
\end{align}
The expansion coefficients \eqref{3nj-Coef} can be expressed in terms of the multivariate $q$-Racah polynomials \eqref{GR-Racah} as follows:
\begin{align}
\label{Racah-Rahman}
\mathcal{R}_{\bm{m}}^{(\bm{\alpha})}(\bm{n}) =
\delta_{m_{d} n_{d}} \delta_{M_{d-1} N_{d-1}}(-1)^{M_{d-2}}
\sqrt{\frac{\chi^{(\bm{\beta})}(\bm{y})}{\Upsilon_{\bm{\ell}}^{(\bm{\beta})}}} Z_{\bm{\ell}}(\bm{y};\bm{\beta},M;q),
\end{align}
where the connection between the original variables and parameters $\bm{n}, \bm{m}, \bm{\alpha}$ and $\bm{y}, \bm{\ell}, \bm{\beta}, M$ is provided by \eqref{Para}.
\subsection{Duality}
The orthogonality relation \eqref{Ortho-a} for the interbasis expansion coefficients $\mathcal{R}_{\bm{m}}^{(\bm{\alpha})}(\bm{n})$ is equivalent to the orthogonality relation for the multivariate Gasper--Rahman $q$-Racah polynomials \eqref{GR-Racah}. The second orthogonality relation \eqref{Ortho-b} can be explained through a duality relation satisfied by the coefficients $\mathcal{R}_{\bm{m}}^{(\bm{\alpha})}(\bm{n})$.

Let us define dual indices $\bm{\tilde{\ell}}$, variables $\bm{\tilde{y}}$, and parameters $\bm{{\tilde{\beta}}}$ by
\begin{align}
\begin{aligned}
\label{Dual-Parameters}
&\tilde{\ell}_j=y_{s+2-j}-y_{s+1-j}, \qquad j=1,\dots, s,
\\
&\tilde{y}_j=M-L_{s+1-j},\qquad j=1,\dots, s,
\\
&{\tilde{\beta}}_0 =\beta_0,
\\
&{\tilde{\beta}}_j=\beta_0-\beta_{s+2-j}-2M+1, \qquad j=1,\dots, s+1.
\end{aligned}
\end{align} 
\begin{Proposition}
\label{Duality}
The map 
\begin{equation}\label{inv}
(\bm{\ell},\bm{y},\bm{\beta},M)\mapsto (\bm{\tilde{\ell}},\bm{\tilde{y}},\bm{{\tilde{\beta}}},M),
\end{equation}
is an involution. Moreover, the $q$-Racah polynomials \eqref{GR-Racah} satisfy the following duality relation
\begin{align}
& \frac{Z_{\bm{\ell}}(\bm{y};\bm{\beta}, M;q)}{q^{M L_s}(q^{-M};q)_{L_s}(q^{-M-\beta_0})_{L_s}\prod_{j=1}^{s}q^{\ell_j\beta_j/2}(q^{\beta_{j+1}-\beta_{j}};q)_{\ell_j}} \nonumber \\
&\qquad\qquad= \frac{Z_{\bm{\tilde{\ell}}}(\bm{\tilde{y}};\bm{{\tilde{\beta}}},M;q)}{q^{M\tilde{L}_s}(q^{-M};q)_{\tilde{L}_s}(q^{-M-{\tilde{\beta}}_0})_{\tilde{L}_s}\prod_{j=1}^{s}q^{\tilde{\ell}_j{\tilde{\beta}}_j/2}(q^{{\tilde{\beta}}_{j+1}-{\tilde{\beta}}_{j}};q)_{\tilde{\ell}_j}}. \label{polduality}
\end{align}
\end{Proposition}
\begin{proof}
The statement can be deduced from \cite{Iliev2011}, but we provide a direct proof here for the convenience of the reader. 
Applying Sears' transformation formula \cite[page~49, formula (2.10.4)]{Gasper2004}) one can show that 
\begin{equation}\label{1Dtransformation}
r_{n}(x;a,b,c,N;q)=r_{n}(N-x;b,a,-c,N;q).
\end{equation}
Using \eqref{1Dtransformation}, we can rewrite the multivariate $q$-Racah polynomials as follows
\begin{equation}
\label{GR-Racah2}
\begin{split}
&Z_{\bm{\ell}}(\bm{y};\bm{\beta},M;q) = \\
&\prod_{k=1}^{s}
r_{\ell_k}(y_{k+1} - y_{k};\beta_{k+1}-\beta_{k} -1, 2L_{k-1}+\beta_{k}-\beta_{0}-1,-y_{k+1}-L_{k-1}-\beta_{k},y_{k+1}-L_{k-1};q).
\end{split}
\end{equation}
Substituting the dual variables \eqref{Dual-Parameters} into \eqref{GR-Racah2}, one can show that the ${}_4\phi_3$ series in $r_{\ell_k}$ above coincides with the ${}_4\phi_3$ series in  $r_{\tilde{\ell}_{s+1-k}}$ in the dual variables. Thus all ${}_4\phi_3$ terms in 
$Z_{\bm{\ell}}(\bm{y};\bm{\beta},M;q)/Z_{\bm{\tilde{\ell}}}(\bm{\tilde{y}};\bm{{\tilde{\beta}}},M;q)$ cancel. Simplifying and rearranging the remaining products, we obtain equation \eqref{polduality}.
\end{proof}

Using formulas \eqref{R-weight}, \eqref{R-norm} and \eqref{Dual-Parameters} one can check that 
\begin{multline}
\label{dident}
\frac{\Upsilon_{\bm{\ell}}^{(\bm{\beta})}\,\chi^{(\widetilde{\bm{\beta}})}(\bm{\tilde{y}})}{\left[q^{M L_s}(q^{-M};q)_{L_s}(q^{-M-\beta_0})_{L_s}\prod_{j=1}^{d}q^{\ell_j\beta_j/2}(q^{\beta_{j+1}-\beta_{j}};q)_{\ell_j}\right]^2} 
\\
= q^{-M(\beta_{s}+{\tilde{\beta}}_{s}-\beta_0+M)}\frac{(q^{\beta_{s+1}};q)_{2M}(q^{{\tilde{\beta}}_{s+1}};q)_{2M}}{\left[(q;q)_{M}(q^{\beta_{0}+1};q)_M\right]^2}.
\end{multline}
Since the right-hand side of \eqref{dident} is invariant under the involution  \eqref{inv}, the last identity combined with Proposition~\ref{Duality} shows that $q$-Racah polynomials, viewed as polynomials in the indices, are also orthogonal with respect to an appropriate multivariate $q$-Racah weight \eqref{R-weight}. More precisely, the following statement holds.
\begin{Corollary}
With the notations above we have
\begin{equation}\label{dualorth}
\sqrt{\frac{\chi^{(\bm{\beta})}(\bm{y})}{\Upsilon_{\bm{\ell}}^{(\bm{\beta})}}}Z_{\bm{\ell}}(\bm{y};\bm{\beta},M;q)
=\sqrt{\frac{\chi^{(\tilde{\bm{\beta}})}(\bm{\tilde{y}})}{\Upsilon_{\bm{\tilde{\ell}}}^{(\bm{{\tilde{\beta}}})}}}Z_{\bm{\tilde{\ell}}}(\bm{\tilde{y}};\bm{{\tilde{\beta}}},M;q).
\end{equation}
\end{Corollary}
The last equation combined with \eqref{Racah-Rahman} and \eqref{Dual-Parameters} relates the dual orthogonality relation \eqref{Ortho-b} to the orthogonality of the Gasper--Rahman $q$-Racah polynomials.

\section{$q$-Deformed Calogero--Gaudin systems}\label{se4}
In this section, it is shown that the bases of multivariate functions $\Psi_{\bm{n}}^{(\bm{\alpha})}(\bm{y})$ and $\Xi_{\bm{m}}^{(\bm{\alpha})}(\bm{y})$ constructed in Section \ref{se2} are in fact wavefunctions for quantum $q$-deformed Calogero--Gaudin superintegrable systems.

Let us return to the realization \eqref{Other-Realization} of the quantum algebra $su_q(1,1)$; i.e. we take
\begin{align}
\label{Real}
q^{A_0^{(i)}} = q^{y_i + (\alpha_i+1)/2},\qquad A_{+}^{(i)} = \sqrt{\sigma_{y_i+1}^{(\alpha_i)}} T_{y_i}^{+},
\qquad 
A_{-}^{(i)} = \sqrt{\sigma_{y_i}^{(\alpha_i)}} T_{y_i}^{-},
\end{align}
where $T_{y_i}^{\pm} f(y_i)=f(y_i \pm 1)$ is the discrete shift operator in the variable $y_i$, and where $\sigma_{n}^{(\alpha)}$ is given by \eqref{Matrix-Elements}. Following the coproduct construction \eqref{Gen-Rep}, one has $su_q(1,1)$ realizations $su_q^{S}(1,1)$ associated to each set $S= [i;j]$ acting on the variables $y_i,\ldots, y_{j}$. These realizations read
\begin{align*}
A_0^{[i;j]} = \sum_{k=i}^{j}A_0^{(k)}, \qquad A_{\pm}^{[i;j]} = \sum_{k=i}^{j} q^{\sum_{\ell = i}^{k-1}A_0^{(\ell)}}A_{\pm}^{(k)}.
\end{align*}
For a given value of $d$, the ``full'' $su_q(1,1)$ realization corresponds to the set $S=[1;d]$. To each set $S=[i;j]$, one has the intermediate Casimir operators
\begin{align*}
\Gamma^{[i;j]} = \frac{q^{-1/2}q^{A_0^{[i;j]}} + q^{1/2}q^{-A_0^{[i;j]}}}{(q^{1/2}-q^{-1/2})^2} - A_{+}^{[i;j]} A_{-}^{[i;j]} q^{1-A_0^{[i;j]}},
\end{align*}
as per \eqref{Intermediate-Casimir}. Through \eqref{Real}, the intermediate Casimir operators $\Gamma^{[i;j]}$ are easily converted to concrete $q$-difference operators acting on the variables $y_i,\ldots, y_{j}$.  For a given $d$, we define the Hamiltonian
\begin{align}
\label{Hamiltonian}
H = \Gamma^{[1;d]}.
\end{align}
The Hamiltonian \eqref{Hamiltonian} corresponds to a $q$-deformed quantum Gaudin-Calogero system in $(d-1)$ dimension. These systems have been discussed in \cite{Musso2000} in particular. Their integrability was shown, and the eigenvalues and a set of eigenvectors were obtained. In the present approach, it is clear that the Hamiltonian \eqref{Hamiltonian} is in fact superintegrable. Indeed, the elements of the two commutative subalgebras $\langle \Gamma^{[1;2]}, \Gamma^{[1;3]}, \cdots, \Gamma^{[1;d-1]} \rangle$ and $\langle \Gamma^{[2;3]}, \Gamma^{[2;4]}, \ldots, \Gamma^{[2;d]}\rangle$, together with the Hamiltonian $H$, form a set of $2d-3$ algebraically independent symmetries of $H$.
Furthermore, the bases $\Psi_{\bm{n}}^{(\bm{\alpha})}(\bm{y})$ and $\Xi_{\bm{m}}^{(\bm{\alpha})}(\bm{y})$ are wavefunctions for this Hamiltonian satisfying the eigenvalue equations
\begin{align*}
H \,\Psi_{\bm{n}}^{(\bm{\alpha})}(\bm{y}) = E_{N_{d}} \Psi_{\bm{n}}^{(\bm{\alpha})}(\bm{y}),\qquad
H \,\Xi_{\bm{m}}^{(\bm{\alpha})}(\bm{y}) = E_{M_{d}} \Xi_{\bm{m}}^{(\bm{\alpha})}(\bm{y}),
\end{align*}
with energies $E_{N} = \gamma(2 N + A_{d} + d -1)$ as per \eqref{Eigen} and \eqref{eigen-2}. The multivariate $q$-Racah polynomials then correspond to the connection coefficients between two bases for the eigenstates of the quantum Calogero--Gaudin model \eqref{Hamiltonian}.
\section{Conclusion}\label{se5}
Summing up, we have constructed bases of multivariate $q$-Hahn and $q$-Jacobi polynomials in the framework of multifold tensor product representations of the quantum algebra $su_q(1,1)$. We have shown that the Gasper--Rahman multivariate $q$-Racah polynomials arise as the connection coefficients between these bases of $q$-Hahn and $q$-Jacobi polynomials, and we have provided an interpretation for these polynomials in terms of special $3nj$-coefficients for $su_q(1,1)$. Lastly, we have explained how the $q$-Hahn bases can be interpreted as wavefunctions for $q$-deformed quantum Calogero--Gaudin superintegrable systems of arbitrary dimension, and we have given its symmetries in terms of intermediate Casimir operators.

It would be of great interest in the future to determine the invariance algebra generated by the symmetries of the Hamiltonian \eqref{Hamiltonian}. In addition to providing a $q$-extension of the generalized Racah algebra obtained in \cite{DeBie2016, Iliev2016}, it would give an algebraic framework for the bispectral operators of the Gasper--Rahman multivariate $q$-Racah polynomials constructed in \cite{Iliev2011}. It would also allow to define the higher rank Zhedanov algebra. We note that a different interpretation of the bispectral operators for the Gasper--Rahman multivariate $q$-Racah polynomials within the context of the $q$-Onsager algebra was given in \cite{Baseilhac2015}.

It is natural to ask for extensions of the results in the present paper for the multivariate $q$-Racah functions and their bispectral and duality properties established in \cite{GI2011}.
Also of interest would be to consider a similar approach based on the $osp_q(1,2)$ quantum superalgebra. This should give rise to superintegrable Calogero--Gaudin models with $osp_q(1,2)$ symmetry, as introduced in \cite{Musso2005}. It would also provide a framework to obtain a $q$-deformation of the higher rank Bannai--Ito algebra; see \cite{DeBie2016a, Genest2014c, Genest2016}. We plan to report on these questions in the near future.
\section*{Acknowledgments}
This work was initiated during a visit of VXG at Georgia Institute of Technology. VXG holds a postdoctoral fellowship from the Natural Science and Engineering Research Council of Canada (NSERC). The research of PI is supported in part by the Simons Foundation. 
The research of LV is supported in part by NSERC.
\section*{References}

\end{document}